\documentclass[preprint,12pt]{elsarticle}

\usepackage{amssymb}
\usepackage{amsthm}
\usepackage{algorithm}
\usepackage{algpseudocode}
\usepackage{booktabs}
\usepackage{mathtools}
\usepackage{graphicx}
\usepackage{microtype}
\usepackage{url}

\DeclareMathOperator{\cone}{cone}
\DeclareMathOperator{\dual}{dual}

\DeclareMathOperator{\col}{col}
\DeclareMathOperator{\rank}{rank}

\newcommand{\Psucc}{P_{\mathrm{succ}}}
\newcommand{\Npool}{N_{\mathrm{pool}}}

\newtheorem{theorem}{Theorem}[section]
\newtheorem{proposition}[theorem]{Proposition}

\newtheorem{corollary}[theorem]{Corollary}
\theoremstyle{remark}
\newtheorem{remark}[theorem]{Remark}
\newtheorem{conjecture}[theorem]{Conjecture}
\theoremstyle{plain}

\newcommand{\R}{\mathbb{R}}
\newcommand{\Rp}{\R_{\ge 0}}
\newcommand{\Ao}{A_{\mathrm{orth}}}
\newcommand{\Aoo}{A_{\mathrm{orth}1}}
\newcommand{\Idn}{I}

\journal{Linear Algebra and its Applications}

\begin{document}

\begin{frontmatter}

\title{Exact Nonnegative Matrix Factorization via Cone-Ray Witnesses:
Certificates, a One-Sided Solver, and a Findability Phase
Transition\tnoteref{ack}}
\tnotetext[ack]{Software drafts, prose, and the Monte-Carlo characterisation
pipeline were prepared with the assistance of Anthropic's Claude (Opus~4.8);
the author bears full responsibility for the content.}

\author{Mithil Ramteke\fnref{gm}}
\ead{mramteke@qti.qualcomm.com}
\fntext[gm]{Alternate email: \texttt{mithilr@gmail.com}.}
\address{Qualcomm India Private Limited, India}

\begin{abstract}
We study exact nonnegative matrix factorization (NMF) of small
exact-rank-$r$ matrices through the polyhedral cones of nonnegative
preimages of the truncated SVD. Restricting each factor to an $r$-subset of
a cone's extreme rays collapses the factorization constraint to the
entrywise nonnegativity of a single $r\times r$ \emph{witness matrix};
feasibility of the witness is a \emph{certificate} that an exact size-$r$ NMF
exists, decided in one matrix inverse. A single-coupling completeness theorem
shows every size-$r$ NMF is representable this way, for every $m$, and a
one-sided relaxation gives a closed-form solver that provably dominates the
two-sided witness, with an explicit cone-geometric criterion for its strict
gain. Our main result concerns \emph{findability}: at a fixed search budget, witness
recoverability undergoes a sharp conic phase transition whose width collapses
to a step as $r$ grows. We rule out both a universal-constant explanation
(Goemans--Williamson) and a statistical-dimension explanation --- the cones'
statistical dimension is flat across the transition, while the intrinsic-volume
profile's \emph{variance} tracks it. The transition is not an existence
boundary: by completeness a ray-economical witness always exists and persists
past the threshold, so what decays with $m$ is its \emph{density} among
$r$-subsets. The threshold is therefore budget-relative --- moving
logarithmically as the candidate pool grows --- and combinatorial rather than
smooth-conic; a single budget law bridges this finite-budget threshold to the
budget-free existence limit. Obtuseness ranking surfaces a feasible witness near the top of
astronomically many subsets, and a two-sided union of one-sided relaxations
returns an exact, machine-precision factorization with a decidable certificate.
That search is itself budget-limited, so its recovery obeys the very transition
we study: competitive with heuristics on small instances, it demands a budget
growing with $m$ as the witness density decays, and a well-initialized
coordinate-descent solver --- not subject to this combinatorial threshold ---
recovers more at larger sizes and in less time. Beyond the exact-rank case, the
same coupling yields a decidable certificate for the complementary \emph{gap}
regime $r_+ > r$: infeasibility of its linear relaxation proves that no size-$r$
factorization exists, with an explicit dual witness. Our contribution is therefore the
certificate and the transition it exposes, not a faster general-purpose solver;
the unifying theme is that the factorization always \emph{exists} and is
\emph{representable}, while ray-economical \emph{findability}, relative to a
search budget, is what transitions.
\end{abstract}


\end{frontmatter}

\section{Introduction}\label{sec:intro}
Nonnegative matrix factorization (NMF) writes $A \approx W H^{\!\top}$ with
$W, H \ge 0$. We study its \emph{exact} form --- $A = W H^{\!\top}$ with inner
dimension $r$ --- for small exact-rank matrices, without assuming
separability. Exact NMF is NP-hard in general \cite{Vavasis2009} but
polynomial for fixed small $r$ \cite{AroraGeKannanMoitra2012}, which motivates
practical certified methods in that regime. Our vehicle is the cone-ray
pipeline: truncate the SVD, form the polyhedral cone of nonnegative preimages
of each factor's column space, enumerate its extreme rays by the double
description method \cite{motzkin1953double, fukuda1996double}, and search for a
nonnegative coupling of the rays.

A single distinction organises the paper. For an exact-rank-$r$ input the
size-$r$ factorization always \emph{exists}, and (Theorem~\ref{thm:complete})
is always \emph{representable} in the cone-ray coupling --- for every $m$. What
is neither guaranteed nor monotone is \emph{findability}: whether a
\emph{ray-economical} witness, using only $r$ rays per side, can be found.
Findability is the object that undergoes a phase transition, and pinning down
what governs it is our main contribution.

\paragraph{Contributions.}
\begin{enumerate}
  \item \textbf{Certified witness.} Under uniform support the constraint
        collapses to the entrywise nonnegativity of an $r\times r$ witness
        (Proposition~\ref{prop:Mt}); its feasibility \emph{certifies} that an
        exact size-$r$ NMF exists, in one $r\times r$ inverse.
  \item \textbf{Completeness (single coupling).} Theorem~\ref{thm:complete}:
        every size-$r$ NMF is representable in the single-coupling form
        $R\,\Delta\,T^{\!\top} = \Idn_r$, $\Delta \ge 0$, for every $m$.
  \item \textbf{One-sided solver.} Theorem~\ref{thm:onesided}: a closed-form
        solver that provably dominates the two-sided witness, with an explicit
        cone-geometric criterion (triangulation straddling) for the strict
        gain.
  \item \textbf{Findability phase transition.} At a fixed search budget, a
        sharp conic transition (\S\ref{sec:phase}) whose width steps down with
        $r$; we rule out both the Goemans--Williamson constant and the
        statistical dimension, show the intrinsic-volume-profile variance
        tracks it, prove that a simplicial recovery cone always yields a witness
        (Proposition~\ref{prop:r2}, so the phenomenon is absent for $r \le 2$
        and confined to $k_1 > r$ otherwise), and conjecture a combinatorial
        density-decay law whose threshold is budget-relative (shifting
        logarithmically with the candidate pool).
  \item \textbf{Obtuseness anti-blowup + a certifying two-sided union.} Ranking
        surfaces a witness at median list position $\sim\!4$ of astronomically
        many subsets (\S\ref{sec:search}); the two-sided union of one-sided
        relaxations returns exact, machine-precision factorizations with a
        decidable certificate (\S\ref{sec:exp}). Being a budget-limited witness
        search, its recovery is itself governed by the transition of
        \S\ref{sec:phase}: it is competitive in the findable regime (small $m$,
        light tails) but, as the density decays, is overtaken by well-initialized
        coordinate descent, which faces no such combinatorial threshold. The
        union's role is to certify, not to compete as a general-purpose solver.
  \item \textbf{The gap regime: a certified boundary.} For nonnegative rank
        $r_+ > r$ no valid gauge exists, and neither the closed-form witness nor
        a Stiefel gradient search finds one (\S\ref{sec:gap}); yet the same
        completeness theorem \emph{certifies} it --- infeasibility of the
        coupling's linear relaxation proves $r_+ > r$ with an explicit dual
        witness (Corollary~\ref{cor:gapcert}), decidably, and fires on the
        classical rank-$3$ gaps. We also flag a random construction that
        spuriously appears to be a gap.
\end{enumerate}

\paragraph{Positioning.}
The idea of factoring through the intersection of the SVD subspace with the
nonnegative orthant, reading off factors from the resulting cone's extreme
rays, is due to the ``nonnegative canonical / subspace edge'' line of
Pimentel-Alarc\'on and coauthors
\cite{pimentel2024canonical, nguyen2025conecollapse}, whose edges coincide
with our cone rays; we concede that scaffolding. Our contribution is what it
supports for the \emph{exact} problem: the certified witness, the
completeness and one-sided theorems, and the findability transition. Against
heuristics (\S\ref{sec:exp}) the witness is a certifier, not a competitor.

\section{Preliminaries}\label{sec:prelim}
For $A \in \R^{m \times n}_{\ge0}$ of rank $r$ (the exact-rank-$r$ setting, where
a size-$r$ NMF exists iff the nonnegative rank is also $r$), the thin SVD
$A = U S V^{\!\top}$ gives the balanced half-factors
\[
  \Ao = U_{:,1:r}\sqrt{S_{1:r,1:r}}, \qquad
  \Aoo = V_{:,1:r}\sqrt{S_{1:r,1:r}},
\]
so $A = \Ao\Aoo^{\!\top}$ and any $W = \Ao Q$, $H = \Aoo P$ with
$Q P^{\!\top} = \Idn_r$ reconstructs $A$ exactly. The set
$\{x \in \R^r : \Ao x \ge 0\}$ is a polyhedral cone; its extreme rays form the
$r \times k_1$ matrix $R$, obtained from the $m$ facet normals by the double
description method (DDM \cite{motzkin1953double, fukuda1996double}, realised
via Fukuda's \texttt{cddlib}), and analogously $T \in \R^{r \times k_2}$ for
$\Aoo$. A column of $Q$ (resp.\ $P$) yields nonnegative $W$ (resp.\ $H$) iff it
lies in $\cone(R)$ (resp.\ $\cone(T)$). Feasibility subproblems use the
standard slack-LP relaxation of a generic feasibility system
$\{Gz = h,\ z \ge 0\}$ --- introduce $t_\pm \ge 0$, minimise
$\mathbf 1^{\!\top}(t_+ + t_-)$ subject to $Gz + t_+ - t_- = h$,
$z, t_\pm \ge 0$ --- whose optimum is zero iff the system is feasible.

The ray counts $k_1, k_2$ are the pipeline's cost driver. By McMullen's upper
bound the number of extreme rays of a cone cut by $m$ facets in $\R^r$ can grow
like $\binom{m}{\lfloor r/2\rfloor}$, and empirically both $k_1$ and $k_2$ range
from $\sim\!16$ at $r = 4$ to $\sim\!50$ at $r = 6$ for $m = 10$, so the number
of witness $r$-subset pairs $\binom{k_1}{r}\binom{k_2}{r}$ runs from
$\sim\!3\times10^{6}$ at $r = 4$ to $\sim\!2\times10^{14}$ at $r = 6$. Exhaustive enumeration is hopeless; the obtuseness ranking of
\S\ref{sec:search} is what makes the search tractable, and the DDM step itself
is what eventually caps the reachable problem size (\S\ref{sec:exp}).

This cost is intrinsic to the problem, not to the choice of SVD. Any rank
factorization $A = BC^{\!\top}$ yields half-factors that differ from
$\Ao,\Aoo$ by an invertible $r\times r$ change of basis $G$, which sends the
cone $\{x:\Ao x\ge0\}$ to $G^{-1}\{x:\Ao x\ge0\}$ and so preserves its face
lattice exactly. The ray counts $k_1,k_2$, the witness, its density, and the
transition of \S\ref{sec:phase} are therefore invariant to the factorization
used (SVD, LU, QR); the enumeration cost is a function of $(m,r)$ and the cone
geometry alone. We take the truncated SVD for its conditioning, but a cheaper
triangular factor cannot reduce the DDM cost --- the same invariance that makes
the transition a clean, basis-free object forbids a factorization-based
speedup.

\section{The Cone-Ray Witness and Its Certificate}\label{sec:witness}
Recall from \S\ref{sec:prelim} the balanced half-factors $\Ao,\Aoo$ (with
$A = \Ao\Aoo^{\!\top}$) and the ray matrices $R \in \R^{r\times k_1}$,
$T \in \R^{r\times k_2}$ holding the extreme rays of $\{x:\Ao x\ge0\}$ and
$\{y:\Aoo y\ge0\}$. A size-$r$ factorization $W=\Ao Q$,
$H=\Aoo P$ reconstructs $A$ iff $Q P^{\!\top}=\Idn_r$, and is nonnegative iff
the columns of $Q$ (resp.\ $P$) are nonnegative combinations of the rays $R$
(resp.\ $T$).

\paragraph{The uniform-support witness.}
Restrict each factor to a single $r$-subset of rays: $S\subset[k_1]$ on the
$W$-side and $K\subset[k_2]$ on the $H$-side, with $R_S, T_K$ the
corresponding $r\times r$ ray matrices.

\begin{proposition}[Witness collapse and certificate]\label{prop:Mt}
For invertible $R_S, T_K$, fixing the $H$-side coefficient to $\Idn_r$
forces the $W$-side coefficient to the \emph{witness matrix}
$M = R_S^{-1}(T_K^{\!\top})^{-1}$. The pair $(S,K)$ yields the
\emph{uniform-support} exact size-$r$ NMF --- the one whose $W$-columns lie in
$\cone(R_S)$ --- iff $M\ge0$ entrywise; hence $M\ge0$ is a \emph{sufficient
certificate} that an exact size-$r$ NMF of $A$ exists, decided by one
$r\times r$ inverse and a sign check.
\end{proposition}
\begin{proof}
Fix the $H$-side coefficient to $\Idn_r$, so $H = \Aoo T_K$, and seek a
$W$-side coefficient $M\in\R^{r\times r}$ giving $W = \Ao R_S\,M$. Then
$W H^{\!\top} = \Ao\,(R_S\,M\,T_K^{\!\top})\,\Aoo^{\!\top}$, which equals
$A = \Ao\Aoo^{\!\top}$ iff $R_S\,M\,T_K^{\!\top} = \Idn_r$ (as $\Ao,\Aoo$ have
full column rank); invertibility of $R_S,T_K$ then forces
$M = R_S^{-1}(T_K^{\!\top})^{-1}$. The factor $H = \Aoo T_K \ge 0$ automatically,
since the columns of $T_K$ are extreme rays of $\{y:\Aoo y\ge0\}$; and because
the columns of $R_S$ are extreme rays of $\{x:\Ao x\ge0\}$ --- so each satisfies
$\Ao (R_S)_{:,i} \ge 0$, whence $\Ao R_S \ge 0$ --- the factor
$W = (\Ao R_S)\,M \ge 0$ once $M\ge0$.
Equivalently, the uniform-support coupling $\Delta$ (supported on the $(S,K)$
block, where it equals $M$) is nonnegative iff $M \ge 0$. Hence $(S,K)$ yields
the exact size-$r$ NMF $(W,H) = (\Ao R_S M,\ \Aoo T_K)$ iff $M \ge 0$, decided by
one $r\times r$ inverse and a sign check.
The converse holds only in this uniform-support sense: $W = (\Ao R_S)M \ge 0$
requires each $R_S M_{:,\ell} \in \cone(R)$, and a point of $\cone(R)$ may have
negative coordinates in the basis $R_S$ when it is reached through extreme rays
outside $S$. Such a pair still factorizes $A$, but not with $W$-columns in
$\cone(R_S)$; $M \ge 0$ is therefore sufficient, and necessary only for the
uniform-support solution. Sufficiency is what the certificate uses.
\end{proof}

Feasibility of $M$ is a stronger object than a small residual: it
\emph{proves} an exact size-$r$ NMF exists (\S\ref{sec:exp} contrasts this
with the numerical residual of a heuristic).

\paragraph{The single-coupling formulation.}
Lifting the two-sided pair to one nonnegative matrix, seek
$\Delta\in\Rp^{k_1\times k_2}$ with
\begin{equation}\label{eq:constraint}
  A \;=\; \Ao\,(R\,\Delta\,T^{\!\top})\,\Aoo^{\!\top},
  \qquad\text{i.e.}\qquad R\,\Delta\,T^{\!\top}=\Idn_r,\ \ \Delta\ge0.
\end{equation}
Each nonzero $\Delta_{ij}$ contributes a nonnegative rank-one term
$\Delta_{ij}(\Ao r_i)(\Aoo t_j)^{\!\top}$, so any rank-$s$ nonnegative
factorization of $\Delta$ yields an NMF of $A$ of size $s$; in particular
$\rank_+(A) \le \rank_+(\Delta)$.

\begin{theorem}[Completeness]\label{thm:complete}
If $A = W H^{\!\top}$ with $W, H \ge 0$ of inner dimension $r$, then
there exist $\Phi \in \Rp^{k_1 \times r}$, $\Psi \in \Rp^{k_2 \times r}$
such that $\Delta = \Phi\,\Psi^{\!\top} \ge 0$ satisfies
$R\,\Delta\,T^{\!\top} = \Idn_r$ and reconstructs $A$. Thus every size-$r$
NMF is representable in~\eqref{eq:constraint}, for every $m$. (Here $\Phi,\Psi$
are the ray-coordinate factors of the coupling, distinct from the $r\times r$
witness $M$ of Proposition~\ref{prop:Mt}.)
\end{theorem}
\begin{proof}
Put $Q = \Ao^{+}W$, $P = \Aoo^{+}H$. Since $A = \Ao\Aoo^{\!\top}
= W H^{\!\top}$ and $\Ao,\Aoo$ have full column rank, $Q P^{\!\top} = \Idn_r$.
Because $\col W = \col A = \col\Ao$ (each of rank $r$), $\Ao Q = \Ao\Ao^{+}W = W$;
thus $\Ao Q_{:,\ell} = W_{:,\ell} \ge 0$, so each column of $Q$ lies in
$\{x : \Ao x \ge 0\}$, which equals $\cone(R)$ because $\Ao$ has full column
rank, making the cone pointed and hence the conic hull of its extreme rays;
so giving $Q = R\,\Phi$ with $\Phi \ge 0$; likewise
$P = T\,\Psi$, $\Psi \ge 0$. Then $R (\Phi\,\Psi^{\!\top}) T^{\!\top}
= (R\,\Phi)(T\,\Psi)^{\!\top} = Q P^{\!\top} = \Idn_r$, and
$\Delta = \Phi\,\Psi^{\!\top} \ge 0$.
\end{proof}

\begin{remark}
The proof is $m$-independent: the size-$r$ factorization is always
representable, even far past the phase transition of
Section~\ref{sec:phase}. What transitions is not existence but
ray-economical \emph{findability}. Verified numerically to residual
$\sim 10^{-15}$.
\end{remark}

\begin{corollary}[Decidable gap certificate]\label{cor:gapcert}
Relax the coupling in~\eqref{eq:constraint} from $\Delta=\Phi\Psi^{\!\top}$ to an
unconstrained $\Delta\ge0$, giving the linear feasibility program: find
$\Delta\ge0$ with $R\,\Delta\,T^{\!\top}=\Idn_r$. If this program is
\emph{infeasible}, then $A$ has no size-$r$ nonnegative factorization; since
$\rank A=r$, equivalently $\rank_+(A)>r$.
\end{corollary}
\begin{proof}
Contrapositive of Theorem~\ref{thm:complete}: any size-$r$ NMF supplies a
feasible $\Delta=\Phi\Psi^{\!\top}\ge0$, in particular a feasible point of the
program. As $\rank A=r$, a size-$r$ NMF exists iff $\rank_+(A)=r$.
\end{proof}
\begin{proposition}[Dual certificate]\label{prop:gapdual}
The program of Corollary~\ref{cor:gapcert} is infeasible iff there exists
$Y\in\R^{r\times r}$ with $R^{\!\top}YT\le0$ entrywise and
$\operatorname{tr}(Y)>0$. Such a $Y$ separates $\Idn_r$ from
$\cone\{r_it_j^{\!\top}\}$: $\langle Y,R\Delta T^{\!\top}\rangle
=\langle R^{\!\top}YT,\Delta\rangle\le0$ for every $\Delta\ge0$, whereas
$\langle Y,\Idn_r\rangle=\operatorname{tr}(Y)>0$; it is verified by one matrix
product and a sign check.
\end{proposition}

\section{Obtuseness-Ranked Search}\label{sec:search}
The candidate space $\binom{k_1}{r}\binom{k_2}{r}$ outgrows any enumeration
budget, so we rank. The witness $M \ge 0$ is more likely for fat
(near-orthogonal) sub-cones, captured by the \emph{obtuseness} score
\begin{equation}\label{eq:obtuseness}
  \mathrm{obtuseness}(S) \;:=\;
  \frac{|\det R_S|}{\prod_{i \in S}\|R_{:,i}\|} \;\in\; [0,1],
\end{equation}
which is $1$ iff the columns $\{R_{:,i}\}_{i \in S}$ are mutually
orthogonal (Hadamard's inequality) and $0$ iff $S$ is rank-deficient. Both
solvers of this paper --- the two-sided witness search below and the one-sided
relaxation of \S\ref{sec:solver} --- walk the top of this ranked list, capping
at $\mathrm{maxTries}$ subsets per side.

\paragraph{Obtuseness tames the blowup.}
The ranking makes the search practical despite the combinatorial candidate
space, measurably so. On the random rank-$r$ generator across a range of $m$
($540$ instances, $r = 4, 5, 6$), the first feasible witness appears at median
list position $\approx 4$ on the $W$-side and $\approx 6$ on the $H$-side,
drawn from a combinatorial space $\binom{k_1}{r}$ reaching $\sim\!2\times10^{7}$
per side (a full witness-pair space of $\sim\!4\times10^{14}$), from which the
ranked pool is sampled. Against a \emph{random} ordering of
the same candidates, the obtuseness order reaches a feasible subset about
$5\times$ sooner (median position $4$ vs.\ $19$), a margin that widens with cone
size --- the ranking, not luck, surfaces the witness.

\paragraph{Augmented alt-LP fallback.}
When no $r$-subset is witness-feasible, we lift the uniform-support
restriction minimally: augment both supports by $a = 2$ rays of maximal
angular separation and solve a short slack-LP alternation for
$\mu, \nu \ge 0$ on the enlarged $r \times (r+2)$ bases. On the $m = n = 10$
Monte Carlo this lifts recovery from the closed-form ceiling to the hybrid
counts of Table~\ref{tab:hybrid}.

\begin{table}[h]\centering
\caption{Success counts (out of $100$) at $m = n = 10$. ``Mt'' columns walk
the top-$N$ obtuseness ranking with the closed-form witness; the
\textbf{hybrid} adds the augmented alt-LP fallback at each visited pair,
breaking the closed-form ceiling that top-$400$ alone
cannot.}\label{tab:hybrid}
\begin{tabular}{lccccc}
\toprule
$r$ & top-$5$ Mt & top-$200$ Mt & top-$400$ Mt & aug2+altLP (top-$1$) & \textbf{hybrid} \\
\midrule
$4$ & $44/100$ & $79/100$ & $79/100$ & $54/100$ & $\mathbf{99/100}$  \\
$5$ & $32/100$ & $85/100$ & $87/100$ & $26/100$ & $\mathbf{95/100}$  \\
$6$ & $\,8/100$ & $58/100$ & $59/100$ & $14/100$ & $\mathbf{75/100}$ \\
\bottomrule
\end{tabular}
\end{table}

\paragraph{Alt-LP mechanics.}
On the augmented bases the witness coefficients become
$\mu, \nu \in \R_{\ge 0}^{(r+a) \times r}$, and the constraint takes the
bilinear form
\begin{equation}\label{eq:joint-aug}
  R_{S'}\,\mu\,\nu^{\!\top} T_{K'}^{\!\top} = \Idn_r,
\end{equation}
$r^2$ scalar equations in $2(r+a)r$ unknowns, with $S', K'$ the $(r+a)$-ray
augmented supports. We solve~\eqref{eq:joint-aug} by alternation: fix $\nu$ and
solve the linear system for $\mu$ by slack-LP (\S\ref{sec:prelim}); fix $\mu$
and solve for $\nu$ likewise. Each subsolve is an LP of size $r^2$ equations in
$(r+a)r$ variables plus $2r^2$ slacks (e.g.\ $16$ equations in
$24 + 32 = 56$ variables at $r = 4, a = 2$). We initialise
$\mu^{(0)} = [\max(M, 0);\,\mathbf 0]$, $\nu^{(0)} = [\Idn_r;\,\mathbf 0]$: since
the fallback is entered precisely when $M \not\ge 0$, the clip $\max(M,0)$ is the
nearest nonnegative point to the (infeasible) uniform-support solution, so the
alternation starts adjacent to it and deviates only as needed. A per-trial cap of $1{,}000$ alt-LP attempts bounds
the worst-case wall time (Algorithm~\ref{alg:hybrid}).

\begin{algorithm}[h]
\caption{Hybrid obtuseness walk with augmented alt-LP fallback.}
\label{alg:hybrid}
\begin{algorithmic}[1]
\State Compute $\Ao, \Aoo$, the ray matrices $R, T$, and their obtuseness
       rankings; set $\mathrm{altAttempts} \leftarrow 0$.
\For{$i = 1, \dots, \mathrm{maxTries}$; \; $j = 1, \dots, \mathrm{maxTries}$}
  \State $S, K \leftarrow$ position-$i$, position-$j$ subsets; form $M = R_S^{-1}(T_K^{\!\top})^{-1}$.
  \If{$M \ge -10^{-8}$ componentwise} \Return $(S, K, M, \Idn_r)$ \Comment{witness path}
  \EndIf
  \If{$\mathrm{altAttempts} < 1{,}000$}
     \State augment $S, K$ by two farthest rays; run $\le 5$ alt-LP rounds
            on~\eqref{eq:joint-aug} from $(\mu^{(0)}, \nu^{(0)})$;
     \State \textbf{if} residual $< 10^{-8}$ \textbf{return} the augmented factors;
            $\mathrm{altAttempts}\mathrel{+}=1$.
  \EndIf
\EndFor
\State \Return \texttt{no\_feasible}.
\end{algorithmic}
\end{algorithm}

\paragraph{Two budgets: walk depth versus pool size.}
Two distinct budgets govern the search. Beyond the top-$200$ the witness
search saturates in \emph{walk depth} (Table~\ref{tab:hybrid}): doubling to top-$400$
adds at most two successes --- but because the sampled pool contains no
feasible pair that the obtuseness order places beyond position $\sim\!200$, not
because feasible pairs are absent from the full candidate space. Enlarging the \emph{pool} is the effective lever: with a
$12\times$ larger sample the witness search recovers far more instances
(\S\ref{sec:phase}, Table~\ref{tab:budget}). Within a fixed pool the residual
failures are instances whose sampled rays lack a near-orthogonal $r$-subset on
one side --- at $m = 10$ the median best obtuseness over the no-feasible trials
falls from $0.74$ at $r = 4$ to $0.33$ at $r = 6$, and in the larger $m = 15$
regime it reaches $0.15$ at $r = 8$ (Table~\ref{tab:m15}) --- but this scarcity
is relative to the pool and recedes as the pool grows. This is the empirical face of the budget-relative
transition of \S\ref{sec:phase}.

\section{A One-Sided-Relaxed Solver}\label{sec:solver}
The two-sided witness fixes $r$ rays on \emph{both} sides. We relax the
$H$-side: fix an $r$-subset $S$ of $R$'s columns and let the $H$-side use
all rays. This has a closed form and provably dominates the
two-sided witness subset-by-subset.

\begin{theorem}[One-sided solver: closed form and dominance]
\label{thm:onesided}
Let $S$ be an $r$-subset of the columns of $R$ with $R_S$ invertible, and
put
\[
  W_S = \Ao R_S, \qquad H_S = \Aoo\,(R_S^{-1})^{\!\top}.
\]
Then $W_S H_S^{\!\top} = A$ and $W_S \ge 0$ for every such $S$, and:
\begin{enumerate}
  \item[(i)] \emph{(closed form)} a nonnegative $D_S$ with
        $R_S D_S T^{\!\top} = \Idn_r$ exists iff $H_S \ge 0$; the $H$-side
        rays $T$ need never be computed;
  \item[(ii)] \emph{(geometric)} equivalently, every row of $R_S^{-1}$
        lies in the cone $C_2 = \{y : \Aoo y \ge 0\}$;
  \item[(iii)] \emph{(dominance)} if the two-sided witness succeeds at
        $(S, K)$ then $H_S \ge 0$, so the one-sided \emph{criterion} holds
        wherever the witness criterion does at the same $S$ --- solver-level
        dominance then follows whenever the one-sided walk also reaches $S$;
  \item[(iv)] \emph{(strict gain)} the one-sided solver succeeds at $S$
        while the witness fails there iff the $r$ rows of $R_S^{-1}$ all
        lie in $C_2$ but no simplicial $r$-ray sub-cone $\cone(T_K)$
        contains all of them.
\end{enumerate}
\end{theorem}
\begin{proof}
$W_S H_S^{\!\top} = \Ao R_S R_S^{-1} \Aoo^{\!\top} = \Ao \Aoo^{\!\top} = A$,
and $\Ao R_S \ge 0$ since the columns of $R_S$ are extreme rays of
$\{x : \Ao x \ge 0\}$. \emph{(i)--(ii)}: $R_S D_S T^{\!\top} = \Idn_r$ means
$D_S T^{\!\top} = R_S^{-1}$, i.e.\ row $a$ of $R_S^{-1}$ equals
$\sum_j (D_S)_{aj}\, t_j$ with $(D_S)_{aj} \ge 0$ --- a nonnegative
combination of the rays $t_j$, hence a point of $\cone(T) = C_2$. Such a
$D_S \ge 0$ exists iff each row of $R_S^{-1}$ lies in $C_2$, i.e.\
$\Aoo (R_S^{-1})^{\!\top} = H_S \ge 0$; and $H_S$ is independent of the
particular $D_S$. \emph{(iii)}: the witness at $(S, K)$ gives
$M = R_S^{-1}(T_K^{\!\top})^{-1} \ge 0$, so $R_S^{-1} = M\,T_K^{\!\top}$ and
each row of $R_S^{-1}$ is a nonnegative combination of the $r$ rays
$\{t_j : j \in K\}$, hence lies in $\cone(T_K) \subseteq C_2$; apply (ii).
\emph{(iv)} is the negation of the witness condition within (ii).
\end{proof}

\begin{remark}
Theorem~\ref{thm:onesided} makes the mechanism explicit: the two-sided
witness forces all $r$ rows of $R_S^{-1}$ into a single facet-simplex
$\cone(T_K)$ of $C_2$, whereas the relaxation only needs them in $C_2$
itself. The gain is exactly the event that these rows \emph{straddle a
triangulation} of $C_2$. The solver is thus a closed-form walk
(Algorithm~\ref{alg:onesided}) requiring only the $W$-side cone.
\end{remark}

\begin{algorithm}[H]
\caption{One-sided-relaxed exact size-$r$ NMF}\label{alg:onesided}
\begin{algorithmic}[1]
\State Compute $\Ao, \Aoo$ and the $W$-side extreme rays $R$ (DDM).
\For{$r$-subsets $S$ in descending obtuseness order}
  \If{$R_S$ invertible and $H_S := \Aoo (R_S^{-1})^{\!\top} \ge 0$}
    \State \Return $W = \Ao R_S$, $H = H_S$.
  \EndIf
\EndFor
\State \Return \texttt{no\_feasible}.
\end{algorithmic}
\end{algorithm}

Empirically the gain is large and grows into the hard regime
(Table~\ref{tab:onesided}, Figure~\ref{fig:two}): at $r = 6, m = 13$ the
witness recovers $1/40$ while the one-sided solver recovers $30/40$, both
at machine precision.

\begin{table}[h]\centering
\caption{Exact size-$r$ recovery on $m \times m$ rank-$r$ matrices
(40 trials/cell). The one-sided solver recovers where the witness has
died.}\label{tab:onesided}
\begin{tabular}{llcc}
\toprule
$r$ & $m$ & witness & one-sided \\
\midrule
6 & 10 & 21/40 & 38/40 \\
6 & 12 & \phantom{0}4/40 & 33/40 \\
6 & 13 & \phantom{0}1/40 & 30/40 \\
\bottomrule
\end{tabular}
\end{table}

\begin{figure}[h]\centering
\includegraphics[width=0.6\linewidth]{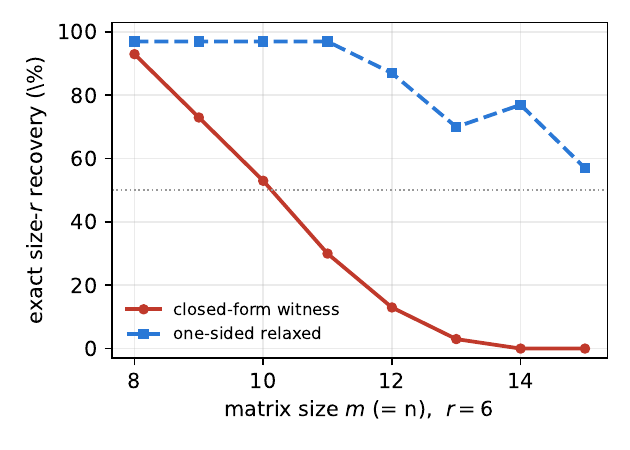}
\caption{Two-threshold contrast at $r = 6$: the one-sided relaxation's
recoverability transition sits far to the right of the witness's.}
\label{fig:two}
\end{figure}

\paragraph{The margin is triangulation straddling.}
The straddling mechanism is measured \emph{directly}, not inferred: both
feasibility tests run off the same dual matrix $R_S^{-1}$, so at each subset
$S$ we record whether its rows lie in a common facet-simplex of $C_2$
(witness) or only in $C_2$ (one-sided). Over $1440$ random
$A = XY^{\!\top}$ (Table~\ref{tab:margin}): (i)~$\Pr[\text{one-sided}] \ge
\Pr[\text{witness}]$ in every cell --- the finite-sample face of
Theorem~\ref{thm:onesided}(iv); (ii)~the per-subset \emph{straddling
fraction} rises from $\approx 0.37$ at $(r,m) = (3,9)$ to $0.73$ at $(5,12)$,
becoming the dominant mode in the harder cells --- consistent with a finer
triangulation of $C_2$, which we do not measure directly; (iii)~the gain
is largest where the witness is weakest (gain $0.57$ at
$\Pr[\text{witness}] = 0.39$, $r = 5$), broadly rising as the witness rate falls,
though not monotonically across the tabulated cells. So the
relaxation does not merely share the easy regime --- it \emph{extends} the
solvable one. A closed-form bound on the gain remains open: it is the
same-cell probability for $r$ coupled dual points, the \emph{same}
random-cone quantity as the threshold of Section~\ref{sec:phase}
(\S\ref{sec:disc}).

\begin{table}[h]\centering
\caption{One-sided margin over random $A = XY^{\!\top}$ ($80$ instances/cell over
$18$ $(r,m)$ cells, $1440$ instances in all; $5$ representative cells shown;
top-$120$ obtuseness-ranked subsets/side). $\mathrm{gain} =
\Pr[\text{one-sided}] - \Pr[\text{witness}]$; the straddling fraction is the
per-subset share of one-sided wins in no common simplex of $C_2$. The gain is
largest in the cells where the witness is weakest.}\label{tab:margin}
\begin{tabular}{ccccc}
\toprule
$r$ & $m$ & $\Pr[\text{witness}]$ & gain & straddling frac. \\
\midrule
3 & \phantom{0}9 & 0.60 & 0.22 & 0.37 \\
4 & 11 & 0.70 & 0.25 & 0.51 \\
5 & 10 & 0.82 & 0.16 & 0.50 \\
5 & 11 & 0.59 & 0.36 & 0.63 \\
5 & 12 & 0.39 & 0.57 & 0.73 \\
\bottomrule
\end{tabular}
\end{table}

\paragraph{The two-sided union.}
Theorem~\ref{thm:onesided} relaxes the $H$-side. By the symmetry of the
coupling, we may equally fix an $r$-subset $K$ of the $H$-side rays $T$ and
relax the $W$-side:
\[
  H_K = \Aoo T_K, \qquad W_K = \Ao\,(T_K^{-1})^{\!\top},
\]
for which $H_K \ge 0$ holds automatically and $W_K H_K^{\!\top} = A$, with
success iff $W_K \ge 0$ (equivalently, each row of $T_K^{-1}$ lies in
$C_1 = \{x : \Ao x \ge 0\}$). Each direction dominates the two-sided witness
(Theorem~\ref{thm:onesided}(iii)), so their \emph{union} --- run the $W$-fixed
relaxation, then the $H$-fixed one, and return whichever succeeds --- dominates
both, in one deterministic pass, reconstructing $A$ exactly from the successful
side. On the uniform generator the two directions largely coincide (the
$W$-side already succeeds), but they diverge on skewed and heavy-tailed inputs,
where the $H$-fixed direction recovers instances the $W$-fixed one misses. The
union is thus the solver we recommend and benchmark in \S\ref{sec:exp}; it is
what lifts recovery above multi-start HALS on average (Table~\ref{tab:crossdist}).

\paragraph{Enlarging the candidate pool.}
The one-sided walk ranks a pool of $r$-subsets --- all $\binom{k_1}{r}$ of
them when that is at most the cap (default $5000$), otherwise a seeded sample of
that size, so the walk is deterministic given $A$ --- and tests the top ones. Because feasible subsets persist but thin out with $m$
(\S\ref{sec:phase}), the recovery lever is the \emph{pool size}, not the walk
depth: enlarging the pool to $60000$ raises one-sided recovery at $r=6$ from
$28\%$ to $77\%$ at $m=17$ (and $78\%\!\to\!98\%$ at $m=13$), at a sublinear
time cost ($12\times$ pool for $\approx\!8\times$ time,
$0.28\mathrm{s}\!\to\!2.2\mathrm{s}$ per instance;
Table~\ref{tab:budget}). The obtuseness key ($|\det R_S|/\prod\|R_{:,i}\|$,
matched to the two-sided witness, where a near-orthogonal $R_S$ makes
$M = R_S^{-1}(\cdot)$ well conditioned) remains a good default here. A key
matched to the \emph{one-sided} geometry does marginally better at ranking:
success at $S$ is $\dual(\cone R_S)\subseteq C_2$, and a smaller dual cone is
enclosed more readily, so ranking by how clustered the dual generators
(the rows of $R_S^{-1}$) are surfaces feasible subsets to median position $0$--$1$
versus obtuseness's $3$--$11$. At a top-$200$ walk both keys reach a feasible subset in
essentially every instance where the pool contains one, so the ranking-key
advantage is slight and the pool size dominates;
we therefore keep obtuseness and expose the pool size as the tunable lever.

\section{The Findability Phase Transition}\label{sec:phase}
On $m \times m$ rank-$r$ matrices ($m \ge r$), at a fixed search budget (the
default pool of $5000$ sampled subsets per side, top-$200$ walk), the
witness-feasibility success rate falls sharply as $m$ grows, and the fall
\emph{sharpens to a step} as $r$ grows (Table~\ref{tab:scaling},
Figure~\ref{fig:scaling}) --- the signature of a conic phase transition
\cite{Amelunxen2014}. The threshold's dependence on that budget is itself the
subject of the density-decay analysis below. We locate the
transition by fitting a logistic
$\Psucc(m) = [1 + e^{(m - m_c)/w}]^{-1}$ to the
$200$-trial success rates on each rank, reading off the $50\%$ threshold $m_c$
and the width $w$. (We reserve $p$ throughout for the per-subset feasibility
\emph{density} introduced below; $\Psucc$ is the solver's \emph{success
probability}. The density-decay analysis then derives the former's control over
the latter.) The width collapses by an order of magnitude across the
range, $w: 3.52 \to 0.24$ as $r: 4 \to 8$ (Table~\ref{tab:scaling}) --- the
transition tightens toward a genuine step. An integer $m$-grid still resolves a
sub-unit width, since the logit of the success rate falls by $1/w$ between
adjacent $m$: bootstrapping the $200$-trial cells puts the $r \ge 5$ widths
within $\pm 15\%$ (at $r = 8$, $0.24 \in [0.20, 0.28]$). The loose entry is
instead $r = 4$, $[2.64, 5.20]$, whose window does not span its own transition;
the order-of-magnitude collapse survives at either extreme. We then rule out
the three natural closed-form explanations. First, the threshold is not a
universal constant: it is $r$- and $m$-dependent, excluding an
approximation-ratio value such as the Goemans--Williamson constant
\cite{Goemans1995} (the $\sim\!87\%$ ``wall'' is not that constant). Second,
it is not set by the extreme-ray count $k_c$, which varies by a factor of two
across the fitted thresholds (Table~\ref{tab:scaling}). Third --- and most
tellingly --- it is not the \emph{statistical dimension} of the cones.

\begin{table}[h]\centering
\caption{Logistic-fitted 50\% threshold $m_c$, transition width $w$, and ray
count $k_c$ at threshold (200 trials/cell), \emph{at the stock search budget}
($5000$-subset pool per side, top-$200$ walk). These $m_c$ are the operating
point of that fixed-budget sampler, not intrinsic constants: they shift right
with the budget (Table~\ref{tab:budget}); the budget-free threshold is the
existence transition of Table~\ref{tab:clump}.}\label{tab:scaling}
\begin{tabular}{cccc}
\toprule
$r$ & $m_c$ & width $w$ & $k_c$ \\
\midrule
4 & 13.47 & 3.52 & 16.0 \\
5 & 11.90 & 1.33 & 29.7 \\
6 & 10.08 & 0.69 & 36.6 \\
7 & \phantom{0}9.47 & 0.46 & 33.5 \\
8 & \phantom{0}9.48 & 0.24 & 28.6 \\
\bottomrule
\end{tabular}
\end{table}

\begin{figure}[h]\centering
\includegraphics[width=0.6\linewidth]{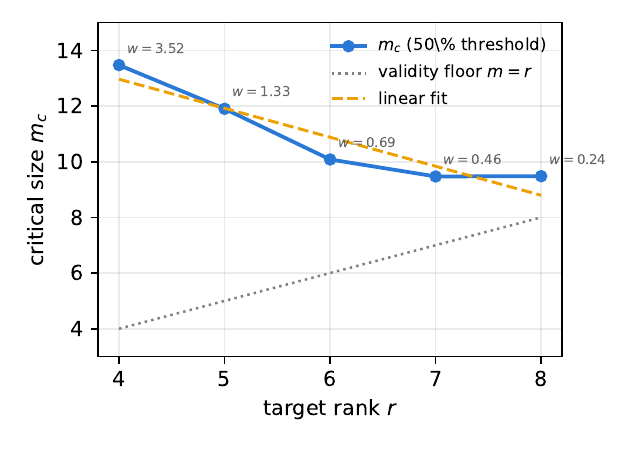}
\caption{Fixed-budget critical size $m_c$ vs.\ $r$ (stock $5000$-pool sampler),
approaching the validity floor $m = r$; annotated transition widths collapse
toward a step. The curve is the sampler's operating boundary; it translates
right as the budget grows (Table~\ref{tab:budget}).}
\label{fig:scaling}
\end{figure}

\paragraph{Ruling out the statistical dimension.}
The natural conic-geometry hypothesis \cite{Amelunxen2014} is that the
transition sits where the statistical dimension
$\delta(C_i) = \mathbb{E}_g\,\|\Pi_{C_i}(g)\|^2$ crosses a critical value,
$g \sim \mathcal N(0, \Idn_r)$. Unlike the canonical cones with closed-form
intrinsic volumes (orthant, second-order, PSD, or the descent cones of
structured norms), our $C_i$ are generic data-dependent polyhedra: their
intrinsic volumes are angle sums over all faces, built from spherical-simplex
(Schl\"afli) solid angles that have no elementary form for $r \ge 3$, so no
analytic $\delta$ is available and we estimate it by Monte-Carlo through
Moreau's decomposition. For $C = \{x : \Ao x \ge 0\}$ the polar cone is
$C^\circ = \{-\Ao^{\!\top}\lambda : \lambda \ge 0\}$, and
$\|\Pi_C(g)\|^2 = \|g\|^2 - \|\Pi_{C^\circ}(g)\|^2$ with
$\Pi_{C^\circ}(g) = -\Ao^{\!\top}\alpha^\star$,
$\alpha^\star = \arg\min_{\alpha \ge 0}\|\Ao^{\!\top}\alpha + g\|^2$ a
nonnegative least-squares solve. Averaging over Gaussian $g$ gives $\delta$
(validated on the orthant: $\delta(\R^6_+) = 2.997$ vs.\ the exact $3$);
Table~\ref{tab:statdim} gives $\delta$ at the fitted thresholds. The same
$\alpha^\star$ yields the full intrinsic-volume profile below, at no extra
cost: by complementary slackness the face on which $g$ projects has dimension
$r - \|\alpha^\star\|_0$. Two
facts refute the hypothesis (Figure~\ref{fig:almt}). $\delta(C_1)$ is nearly
\emph{constant in $m$} across the whole transition --- for $r = 6$ it moves
only $4.84 \to 4.72$ as $m: 8 \to 12$ while the success rate falls
$96\% \to 10\%$ --- so it cannot be the quantity that crosses; and no
normalisation of $\delta_c$ is invariant in $r$. A quantity that barely moves
cannot drive a step transition.

\begin{table}[h]\centering
\caption{Statistical dimension at the fitted threshold $m_c$ (Moreau/NNLS
Monte-Carlo). $\delta(C_1)$ tracks $r$, not the
transition.}\label{tab:statdim}
\begin{tabular}{cccc}
\toprule
$r$ & $m_c$ & $\delta(C_1)$ & $\delta(C_1)+\delta(C_2)$ \\
\midrule
4 & 13.47 & 2.78 & \phantom{0}5.58 \\
5 & 11.90 & 3.70 & \phantom{0}7.42 \\
6 & 10.08 & 4.77 & \phantom{0}9.60 \\
7 & \phantom{0}9.47 & 5.75 & 11.53 \\
8 & \phantom{0}9.48 & 6.81 & 13.59 \\
\bottomrule
\end{tabular}
\end{table}

\begin{figure}[h]\centering
\includegraphics[width=\linewidth]{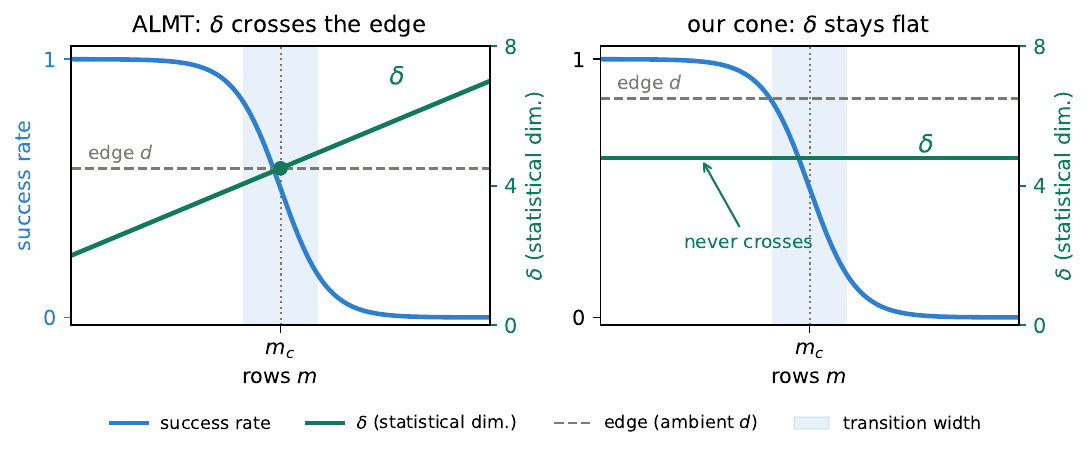}
\caption{Why the statistical dimension does not locate the transition
(schematic). \emph{Left:} in the ALMT picture the transition sits where
$\delta$ rises through the ambient edge $d$ --- the crossing \emph{is} the
threshold, width $O(\sqrt{d})$. \emph{Right:} for the witness cone the success
curve falls with the same shape, but $\delta$ is flat and never crosses. The
width survives as the second-moment effect that Table~\ref{tab:profilevar}
pins at $\mathrm{Var} \approx 0.90$.}\label{fig:almt}
\end{figure}

\paragraph{The full profile: variance controls the width.}
The statistical dimension is only the \emph{mean} of the cone's
intrinsic-volume profile $v_0,\dots,v_r$, where $v_k$ is the probability that
$\Pi_{C_1}(g)$ lands in the relative interior of a $k$-dimensional face
\cite{AmelunxenLotz2017}. The same projection returns the whole profile: by
complementary slackness the landed face has dimension
$r - \|\alpha^{\star}\|_0$ (validated on the orthant's exact binomial
$v_k = \binom{r}{k}2^{-r}$, max deviation $0.004$). Scanning every functional,
only the \emph{variance} $\mathrm{Var} = \sum_k (k-\delta)^2 v_k$ is both
(i)~near-invariant at the threshold, $\mathrm{Var}(C_1) \approx 0.90$ across
$r = 4,\dots,8$ (relative spread $(\max-\min)/\mathrm{mean} \approx 0.06$), and (ii)~genuinely
\emph{sweeping} in $m$ (Table~\ref{tab:profilevar}) --- unlike the solid angle
$v_r \approx 0.23$, flat across the window and characterising nothing. This is
the object conic integral geometry attaches to the transition \emph{width}
\cite{Amelunxen2014}. Since $\mathrm{Var}(r,m)$ is a function of \emph{two}
variables that sweeps in $m$, the equation $\mathrm{Var}(r,m) = 0.90$ is a level
set --- a curve $m(r)$ in the $(r,m)$ plane, not a single value. The point is
that the stock-budget transitions of every rank align on this \emph{one} level;
so solving $\mathrm{Var}(r,m) = 0.90$ for $m$ at each fixed $r$ (the $m$-sweep
supplies the varying solution) recovers $m_c(r)$ to within $\pm 0.3$ for
$r \ge 6$, with an $O(1)$ correction at the smallest ranks. The common level
$0.90$ is itself a stock-budget calibration, not a budget-invariant constant: it
is where the fixed-budget sampler's transition happens to sit, and at a larger
pool the transition slides along the $\mathrm{Var}$-sweep to a different level. The profile pins the \emph{sharpness};
the \emph{location} admits no exact smooth-conic law, reinforcing
Conjecture~\ref{conj:comb}.

\begin{table}[h]\centering
\caption{Face-dimension variance $\mathrm{Var}(C_1)$ at $m_c$, its sweep
across the scanned window, and the $m$ that solves $\mathrm{Var}(r,m) = 0.90$
(the level set the stock-budget transitions align on). Unlike $\delta$, the
variance both sweeps with $m$ and is near-constant at $m_c$.}\label{tab:profilevar}
\begin{tabular}{ccccc}
\toprule
$r$ & $m_c$ & $\mathrm{Var}(C_1)$ at $m_c$ & Var sweep over window
& $m$ at $\mathrm{Var}{=}0.90$ \\
\midrule
4 & 13.47 & 0.88 & $0.81$--$0.88$ & $\approx 14.5$ \\
5 & 11.90 & 0.93 & $0.88$--$0.97$ & $\approx 11.2$ \\
6 & 10.08 & 0.91 & $0.84$--$0.99$ & $\approx 9.9$ \\
7 & \phantom{0}9.47 & 0.90 & $0.79$--$0.98$ & $\approx 9.5$ \\
8 & \phantom{0}9.48 & 0.91 & $0.84$--$1.01$ & $\approx 9.3$ \\
\bottomrule
\end{tabular}
\end{table}

Since $\delta$ is flat while the extreme-ray count grows sharply with $m$
(e.g.\ $k: 18 \to 54$ for $r = 6$), the transition is \emph{combinatorial}:
governed by the configuration of the extreme rays --- the shrinking
\emph{fraction} of $r$-subsets that admit a witness --- not by the solid angle
of the cone.

\paragraph{The same holds for the one-sided solver.}
The one-sided solver of \S\ref{sec:solver} extends the solvable regime, so its
own findability transition sits well to the right of the witness's
($m_c^{\mathrm{OS}} \approx 20$ and $15$ for $r = 5, 6$; broad and beyond
$m = 40$ at $r = 4$, sharpening with $r$ as the witness transition does). Both
diagnostics carry over. The statistical dimension is again flat across it ---
$\delta(C_1) + \delta(C_2) \approx 5.1, 7.2, 9.2$ for $r = 4, 5, 6$, unchanged
in $m$ --- so the $\delta$ ruling-out is solver-independent. And the profile
variance again pins the width, at $\mathrm{Var}(C_1) \approx 1.0$, higher than
the witness's $0.90$ precisely because the transition sits at larger $m$. The
combinatorial-location / conic-width dichotomy thus governs the \emph{effective}
solver, not merely the two-sided witness. As with the witness, this location is
budget-relative: the quoted $m_c^{\mathrm{OS}}$ are for the default pool, and
enlarging it (\S\ref{sec:solver}) shifts them substantially to the right
(Table~\ref{tab:budget}).

\paragraph{The threshold is budget-relative.}
The transition above is measured at one search budget, and it is not an
existence boundary. By completeness (Theorem~\ref{thm:complete}) a witness
exists for every $m$, and direct sampling confirms that ray-economical
witnesses \emph{persist} far past $m_c$: at $r = 6, m = 17$ the feasible
$r$-subsets for the one-sided solver number in the tens of thousands, yet form
a vanishing \emph{fraction} $p(m,r) \approx 8\times10^{-5}$ of all subsets. A
pool of $B$ sampled candidates finds one when $B\,p(m,r) \gtrsim 1$, so the
empirical threshold sits at $p(m_c) \approx 1/B$ and \emph{moves} as the pool
grows. Table~\ref{tab:budget} bears this out: a $12$--$25\times$ larger budget
shifts the $50\%$ point right by roughly $2$ in $m$ on both the witness and the
one-sided solver. Fitting quantifies both halves. First, $\ln p(m,r)$ is
\emph{linear} in $m$ by uniform sampling: the decay is
geometric, $p \approx e^{-c(r)\,m}$ with $c = 0.45$ at $r = 6$
($R^2 = 0.99$, fitted over $m = 10, \dots, 25$) and $c = 0.34$ at $r = 5$
($R^2 = 0.98$, over $m = 12, \dots, 30$), under adaptive sampling that draws until
every $m$ has resolved at least
$30$ feasible subsets. Resolution matters here: a median taken over $0$--$1$ hits
is a floor rather than a measurement, and retaining such points biases the slope
upward (at $r = 6$, to $0.48$); shallower but still-floored sweeps give $0.46$, so
we quote $c$ to two figures only. Second, the induced
threshold $m_c(B)$ --- the $m$ at which a uniform $B$-sampler crosses $50\%$ ---
is linear in $\ln B$: at $r = 6$, over $6$ budgets spanning $2.5$ decades, the
slope is $2.33$ against $1/c = 2.24$ ($4\%$; $R^2 = 0.999$). Since a crossing can
be read only where the density itself is resolvable, the accessible span is set by
sampling depth, and at $r = 5$ only $4$ budgets ($1.5$ decades) survive. At $r = 6$
the span does admit a curvature test: the fitted $(\ln B)^2$ coefficient is
$+0.03$, and the linear fit's residuals ($\pm 0.2$ in $m_c$) are comparable to the
error of interpolating a crossing on an integer $m$-grid. The data are therefore
consistent with a straight log law but do not exclude mild curvature at that
level, and we claim no more. Third, the same rate fixes the transition \emph{width}. Substituting
$p \approx e^{-c\,m}$ into the finite-budget law gives
$\Psucc(m) \approx 1 - \exp(-B\,e^{-c\,m})$, a reversed Gumbel whose midpoint
slope, matched to the logistic's $1/(4w)$, yields
\begin{equation}\label{eq:width}
  m_c = (\ln B)/c(r) + \mathrm{const}, \qquad w = \frac{1}{2\ln 2\;c(r)},
\end{equation}
so the budget \emph{relocates} the transition without sharpening it. We stress
that~\eqref{eq:width} is a \emph{corollary} of the exponential density law, not an
independent finding: once $\ln p$ is linear in $m$, $\Psucc$ is exactly a reversed
Gumbel of scale $1/c$, and both halves of~\eqref{eq:width} follow as algebra. The
numerical check below therefore tests two things only --- that $p$ is close enough
to a pure exponential over the fitted range, and that the constant $1/(2\ln 2)$,
which we obtained by matching midpoint slopes, survives least-squares fitting of a
logistic to a Gumbel over a discrete grid. It is a consistency check of the chain,
not fresh evidence for budget-relativity beyond the linearity of $\ln p$ itself.
With that said, the check passes under uniform sampling across
$r = 4, \dots, 7$: the fitted slope
matches $1/c$ to within $5\%$ (typically $1$--$2\%$), and the width
matches~\eqref{eq:width} to within $20\%$ --- to within $13\%$ for
$r \ge 5$, with $r = 4$ the outlier, its width being the largest and hence the
worst resolved on a discrete $m$-grid. Three scope caveats. Budgets are chosen as
$\ln 2 / p(m^{\ast})$ so the crossing sits inside the sampling window, which
confines each rank to a $\approx\!5$--$7\times$ span in $B$: the log law itself
walks the transition out of any fixed $m$-window, so the width's
budget-independence is tested over that span, not over decades. Ranks whose
density falls below the Monte-Carlo resolution ($p < 3/N_s$) are dropped rather
than fitted, since a median over $0$--$1$ hits is noise. And averaging over the
matrix ensemble broadens the curve a further $1.3$--$2.9\times$ relative to the
\emph{median-density} curve $1 - (1-p_{\mathrm{med}}(m))^{B}$, the broadening
growing as $r$ falls. That reference curve is a construct, not a single cone
tracked across $m$: it feeds the median density of each size's population through
the sampler, discarding matrix-to-matrix spread so that only sampling randomness
remains --- which is precisely the quantity~\eqref{eq:width} predicts. Table~\ref{tab:scaling}'s widths measure a
categorically different quantity and are not comparable to~\eqref{eq:width}. At a
fixed budget the two-sided witness solver is \emph{deterministic} given $A$ --- the
obtuseness ranking either enumerates every $r$-subset or draws a seeded pool, so no
per-instance randomness survives, and repeated runs on one matrix return an
identical verdict. Its success rate at each $m$ is therefore the \emph{fraction of
random matrices} that are solvable, and the fitted width measures the heterogeneity
of the matrix ensemble, not a budget-limited sampling transition;
Equation~\eqref{eq:width} governs the latter. The $2$--$4\times$ gap between the
two is thus a comparison of unlike quantities, and we read no ranking effect --- or
any other single cause --- out of it. (This is consistent with the budget-relativity
of Table~\ref{tab:budget}: the budget is a solver \emph{parameter}, and moving it
moves which matrices are solvable.)
The same rate $c(r)$ thus governs the density decay, the
threshold's budget-dependence, and the transition width. What is intrinsic is therefore not the threshold
\emph{location} but the \emph{decay} $c(r)$ of witness density, which the
intrinsic-volume profile tracks; the statistical dimension, a first-moment solid
angle, is flat and cannot.

\begin{table}[h]\centering
\caption{Recovery (\% of trials) vs.\ search budget at $r = 6$. Enlarging the
candidate pool shifts the transition to the right on both solvers: the
threshold is budget-relative, not an intrinsic constant. One-sided budget is
the pool size (top-$200$ walk); witness budget is pool size / walk depth per
side.}\label{tab:budget}
\begin{tabular}{llccc}
\toprule
solver & budget & \multicolumn{3}{c}{recovery} \\
\midrule
          &                & $m{=}13$ & $m{=}15$ & $m{=}17$ \\
one-sided & pool $5000$    & 78 & 53 & 28 \\
one-sided & pool $20000$   & 95 & 83 & 63 \\
one-sided & pool $60000$   & 98 & 94 & 77 \\
\midrule
          &                & $m{=}10$ & $m{=}11$ & $m{=}12$ \\
witness   & $5000/200$     & 50 & 23 & \phantom{0}7 \\
witness   & $20000/500$    & 85 & 59 & 31 \\
witness   & $60000/1000$   & 96 & 86 & 54 \\
\bottomrule
\end{tabular}
\end{table}

\paragraph{The budget-free limit and witness clumping.}
As $B \to \infty$ the sampler resolves pure \emph{existence}, and existence does
eventually transition: for $r = 4$ --- the only rank with $\binom{k_1}{r}$ small
enough to enumerate every subset exactly --- the fraction of matrices possessing
\emph{any} one-sided-feasible $r$-subset falls through $50\%$ near
$m \approx 46$ (Table~\ref{tab:clump}). But it is \emph{not} the first moment
that governs this. Note $E[N] = \binom{k_1}{r}\,p$ is exact (linearity, no
independence needed) and stays near $10$ across the whole transition; \emph{were
the subsets independent}, existence would be $1 - e^{-E[N]} \approx 1$
(near-certain), yet the true existence is only $0.52$--$0.72$ --- the first
moment over-predicts existence by $\approx 0.4$ and is vacuous. The cause is
severe over-dispersion (variance-to-mean $\sim\!40$--$100$):
feasible subsets occur in correlated \emph{clumps}, families sharing $r - 1$
rays with $\sim\!10$--$16$ members each, so a matrix has either none or dozens.
The number of \emph{clumps}, by contrast, behaves as Poisson --- existence
tracks $1 - e^{-\lambda}$ with $\lambda = \mathbb{E}[\#\mathrm{clumps}]$, and the
transition sits at $\lambda = \ln 2 \approx 0.69$ (Table~\ref{tab:clump}, where
$\mathbb{E}[\#\mathrm{clumps}]$ crosses $0.69$ just as existence crosses $50\%$).
This is the second-moment signature already visible in the profile
\emph{variance}: the same clumping that flattens the mean $\delta$ and voids the
raw count $E[N]$ is what the variance detects. (At $r = 5$ the clumping is
stronger still --- variance-to-mean up to $\sim\!1500$ --- but its existence
transition lies beyond exact enumeration.)

\begin{table}[h]\centering
\caption{Budget-free (existence) transition at $r = 4$, exact by enumerating all
$\binom{k_1}{4}$ subsets ($200$ trials/cell). The first moment is vacuous ($E[N]$ flat,
$1-e^{-E[N]}\approx1$); the \emph{clump count} is the Poisson variable, with
existence $\approx 1 - e^{-\mathbb{E}[\#\mathrm{clu}]}$ and the transition at
$\mathbb{E}[\#\mathrm{clu}] = \ln 2$.}\label{tab:clump}
\begin{tabular}{cccccc}
\toprule
$m$ & existence & $E[N]$ & $\mathrm{Var}/E[N]$ & $\mathbb{E}[\#\mathrm{clu}]$
& $1-e^{-\mathbb{E}[\#\mathrm{clu}]}$ \\
\midrule
28 & 0.72 & 11.7 & \phantom{0}42 & 1.08 & 0.66 \\
34 & 0.67 & \phantom{0}9.5 & \phantom{0}40 & 0.89 & 0.59 \\
40 & 0.59 & \phantom{0}8.2 & \phantom{0}37 & 0.73 & 0.52 \\
46 & 0.52 & \phantom{0}9.9 & \phantom{0}98 & 0.64 & 0.47 \\
\bottomrule
\end{tabular}
\end{table}

\paragraph{A unified budget law.}
The finite-budget shifts (Table~\ref{tab:budget}) and the budget-free existence
transition (Table~\ref{tab:clump}) are two limits of \emph{one} budget-indexed
quantity. A search that samples $B$ candidates succeeds iff it \emph{discovers} a
clump, and the expected number of clumps it discovers is
\[
  \lambda(B) \;=\; \mathbb{E}\Bigl[\,\textstyle\sum_{\text{clumps}}
  \bigl(1 - (1 - s/\Npool)^{B}\bigr)\Bigr],
  \qquad P[\text{success}] \;\approx\; 1 - e^{-\lambda(B)},
\]
with $\Npool = \binom{k_1}{r}$ the candidate-pool size and $s$ a clump's size
(distinct from the decay rate $c(r)$ above). Two \emph{exact} limits follow.
For small $B$ (no clump sampled twice) $\lambda(B) \to B\,p$ with
$p = E[N]/\Npool$,
recovering the finite-budget law $1 - e^{-Bp}$ and threshold $p(m_c) \approx 1/B$;
as $B \to \infty$ every clump is found, $\lambda(B) \to
\mathbb{E}[\#\mathrm{clumps}]$, recovering the existence law with threshold
$\ln 2$. Between them $\lambda(B)$ \emph{saturates} as repeated samples re-hit
known clumps, so the naive first moment $B\,p$ --- blind to this redundancy ---
over-predicts grossly while $\lambda(B)$ does not. Enumerating the $r = 4$ clump
structure and sweeping $B$ from $1$ to exhaustive (Table~\ref{tab:unify}) bears
this out: $1 - e^{-\lambda(B)}$ follows the true sampler success across four
decades of budget --- tightly in the finite-budget regime, and with the same
one-sided sub-Poisson deficit of about $0.05$ at the existence end, consistent
with the clump-exclusion correction discussed below.

\begin{table}[h]\centering
\caption{The unified budget law ($r = 4$, $m = 40$; $150$ cones, exact,
$\mathbb{E}[\#\mathrm{clumps}] = 0.77$, existence $= 0.58$). $\lambda(B)$
interpolates from $B\,p$ (finite budget) to $\mathbb{E}[\#\mathrm{clumps}]$
(existence); $1 - e^{-\lambda(B)}$ tracks the true sampler success across all
budgets (within a small sub-Poisson deficit), while the naive $B\,p$
diverges.}\label{tab:unify}
\begin{tabular}{rcccc}
\toprule
$B$ & actual success & $1 - e^{-\lambda(B)}$ & $\lambda(B)$ & $B\,p$ \\
\midrule
$10^2$ & $0.037$ & $0.037$ & $0.037$ & $0.042$ \\
$10^3$ & $0.209$ & $0.193$ & $0.214$ & $0.420$ \\
$10^4$ & $0.491$ & $0.437$ & $0.575$ & $4.20$ \\
$10^5$ & $0.579$ & $0.531$ & $0.756$ & $42.0$ \\
$10^6$ & $0.580$ & $0.535$ & $0.767$ & $420$ \\
\bottomrule
\end{tabular}
\end{table}

One rigorous anchor bounds the phenomenon: a \emph{simplicial} recovery cone
always yields a witness, so failure requires the cone to have more than $r$
extreme rays.

\begin{proposition}[Simplicial recovery cones are always feasible]\label{prop:r2}
Write $A = \Ao\Aoo^{\!\top} \ge 0$ with $\Ao \in \R^{m\times r}$,
$\Aoo \in \R^{n\times r}$, and $C_1 = \{x\in\R^r : \Ao x \ge 0\}$,
$C_2 = \{y\in\R^r : \Aoo y \ge 0\}$. Then the polar cone satisfies
$C_1^{\circ} \subseteq C_2$. Consequently, whenever $C_1$ is \emph{simplicial}
--- exactly $k_1 = r$ extreme rays $R$ --- the unique $r$-subset witness
$W = \Ao R$, $H = \Aoo (R^{-1})^{\!\top}$ is feasible; a findability failure
requires $k_1 > r$. In particular, at $r = 2$ a pointed cone in $\R^2$ has
exactly two extreme rays, so the witness is \emph{always} feasible and no
transition ever occurs; for $r \ge 3$ the simplicial regime is confined to the
smallest $m$ (empirically $P[k_1 = r] = 1$ only at $m = r$, falling below
$0.4$ already at $m = r+1$), so the transition lives entirely in the $k_1 > r$
regime.
\end{proposition}
\begin{proof}
Let $a_i$ be the $i$-th row of $\Ao$; then
$C_1 = \{x : \langle a_i, x\rangle \ge 0\ \forall i\} = (\cone\{a_i\})^{\circ}$,
so by the bipolar theorem for finitely generated (hence closed) cones,
$C_1^{\circ} = \cone\{a_i\}$. Each $a_i \in C_2$: the $j$-th entry of
$\Aoo\,a_i$ is $\langle \beta_j, a_i\rangle = (\Ao\Aoo^{\!\top})_{ij} = A_{ij}\ge 0$,
where $\beta_j$ is the $j$-th row of $\Aoo$; hence $\cone\{a_i\}\subseteq C_2$, giving
$C_1^{\circ}\subseteq C_2$. When $k_1 = r$ the columns of $R$ are all the extreme
rays, so $\cone(R) = C_1$ and $\dual(\cone R) = C_1^{\circ}$; the rows of $R^{-1}$
are the extreme rays of $C_1^{\circ}$, so $H = \Aoo(R^{-1})^{\!\top}\ge 0$ iff
$C_1^{\circ}\subseteq C_2$, which always holds, while $W = \Ao R \ge 0$ since the
columns of $R$ lie in $C_1$. When $k_1 > r$ a simplicial $r$-subcone
$\cone(R_S)$ is a \emph{proper} subcone of $C_1$, so
$\dual(\cone R_S) \supsetneq C_1^{\circ}$ can exceed $C_2$ --- the source of the
transition.
\end{proof}

Proposition~\ref{prop:r2} also sharpens the exists/findable split: the polar
$C_1^{\circ}\subseteq C_2$ unconditionally (the \emph{full}-cone witness always
works), so what transitions once $k_1 > r$ is strictly the \emph{economy} of
using only $r$ rays. For that regime we conjecture the following.

\begin{conjecture}[Two-scale threshold: density decay and witness clumping]
\label{conj:comb}
Two thresholds govern findability --- the small- and large-$B$ limits of the
budget law $\lambda(B)$ above. \emph{(i) Finite budget.} The density
$p(m,r)$ of witness-compatible $r$-subsets among the extreme rays of the random
cone $\col(W) \cap \R^m_{\ge0}$ decays geometrically in $m$, $p \approx
e^{-c(r)\,m}$, so a $B$-candidate sampler's threshold solves
$p(m_c, r) \approx 1/B$ and $m_c = (\ln B)/c(r) + \mathrm{const}$ grows
logarithmically in $B$ --- both verified by direct fitting above ($R^2 = 0.99$
and $0.999$, over the resolvable budget range). We conjecture that the rate $c(r)$ is set by the angular
(intrinsic-volume) profile of the cone, not by any smooth conic invariant such
as the statistical dimension.
\emph{(ii) Budget-free limit.} Witness-compatible subsets are strongly
over-dispersed, occurring in correlated clumps; the clump count is
\emph{approximately} Poisson --- mildly sub-Poisson, as clumps weakly exclude one
another within a cone (measured dispersion $\mathrm{Var}/\mathbb{E} \approx
0.7$--$1.0$) --- so existence is $1 - e^{-\lambda(m,r)}$ up to that exclusion
correction, with $\lambda = \mathbb{E}[\#\mathrm{clumps}]$ and the intrinsic
transition near $\lambda(m^\ast, r) \approx \ln 2$ --- not $E[N] = 1$, which the
over-dispersion renders vacuous. The two scales draw on \emph{different} probability. The
finite-budget density~(i) is a \emph{single-subset} quantity: it is settled one
subset at a time, fixing $p(m,r)$ and hence the first moment
$E[N] = \binom{k_1}{r}\,p$. It poses an absorption question of Wendel--Cover
\emph{shape} \cite{Wendel1962, Kabluchko2017} --- do $r$ prescribed rays fall
inside a fixed cone? --- but does not inherit their closed form, which rests on
independence: as the probes below show, $p$ does \emph{not} factor over the $r$
rays. It is thus single-subset without being single-ray, and we use
``marginal'' below only in the latter sense. The budget-free limit~(ii) is not
even single-subset: because the indicators \emph{across} subsets are strongly
dependent, $E[N]$ (and any Wendel-type count) is uninformative, and existence is
instead a Poisson approximation for \emph{dependent} events, governed by the
clump structure. Both
belong to random-polytope theory \cite{Barany2008, Reitzner2010} and the
combinatorics of random cones \cite{HugSchneider2016} rather than to conic
integral geometry.
\end{conjecture}

\paragraph{Toward a proof.} Proposition~\ref{prop:r2} settles $r \le 2$. For the
budget-free threshold at $r \ge 3$ the clumping picture suggests a Chen--Stein
route, and we have probed its two premises directly. Write
$N = \sum_S \mathbf{1}[S\text{ feasible}]$ over simplicial $r$-subsets $S$; $N$ is
far from Poisson ($\mathrm{Var}/\mathbb{E}[N]\approx 40$--$1500$,
Table~\ref{tab:clump}), so one passes to the \emph{declumped} count $N^\ast$ (one
representative per Johnson-graph component), for which Chen--Stein gives
$\bigl|\Pr[N^\ast=0]-e^{-\lambda}\bigr| \le b_1 + b_2 + b_3$,
$\lambda=\mathbb{E}[\#\mathrm{clumps}]$. First premise --- \emph{locality} of the
clump --- holds cleanly: the within-cone conditional feasibility
$g(o)=\Pr[S'\text{ feasible}\mid S\text{ feasible},\,|S\cap S'|=o]$, relative to
the baseline $p$, is enhanced $\approx\!40$-fold at $o=r-1$, falls to $\approx\!3$
at $o=r-2$, and to $\approx\!0$ for disjoint subsets, so the dependency
neighbourhood is small and $b_1,b_2$ are controllable. Second premise ---
\emph{decorrelation at range} --- holds only in a corrected form: disjoint
subsets are not independent but weakly \emph{anti}-correlated (clumps exclude one
another), leaving $N^\ast$ mildly sub-Poisson (dispersion $0.7$--$1.0$), so the
honest target is not ``$N^\ast$ Poisson'' but the Chen--Stein \emph{bound} with an
$O(\lambda^2)$ exclusion term --- which quantitatively matches the observed
$\sim\!0.05$--$0.07$ excess of existence over $1-e^{-\lambda}$
(Table~\ref{tab:clump}). Two pieces remain genuinely open: controlling $b_3$ ---
a decorrelation estimate for well-separated subsets, hard because all indicators
read \emph{one} random cone, most plausibly via a Poisson limit for the
extreme-ray process --- and computing $\lambda(m,r)$ from a second-order joint
absorption against the intrinsic-volume profile.

The finite-budget rate $c(r)$ is no simpler. It is fixed by the per-subset
density $p$, itself the joint absorption
$p=\Pr\!\bigl[\text{all } r \text{ dual rays } u_i \text{ of } R_S^{-1}\in C_2\bigr]$,
and direct measurement rules out every \emph{marginal} explanation: $C_2$'s solid
angle is flat in $m$; the dual cone $\dual(\cone R_S)$ widens $\sim\!8\times$ too
slowly; and the single-ray absorption $q=\Pr[u_i\in C_2]\approx 0.5$ --- biased a
stable $\approx\!2.5\times$ into $C_2$ by $C_1^{\circ}\subseteq C_2$
(Prop.~\ref{prop:r2}) --- decays $\sim\!15$--$20\times$ too slowly
($c/c_q\gg r$). The density lies far below the independent product $q^r$, and the
gap \emph{widens} with $m$: the $r$ dual rays are negatively correlated for
simultaneous membership, so $c(r)$ too is an irreducibly joint (multi-ray)
absorption rather than any single solid angle. We leave both rates open.

\section{Empirical Study and Comparison}\label{sec:exp}
Throughout, trials draw $W, H \sim \mathrm{U}[0,1]^{m \times r}$ and set
$A = W H^{\!\top}$ (exact rank-$r$, nonnegative); reconstruction error is
$\|A - \widetilde W \widetilde H^{\!\top}\|_F / \|A\|_F$, at machine precision
whenever the toolkit succeeds (\S\ref{sec:witness}).

\paragraph{Comparison with multi-start heuristics.}
To situate the witness against the workhorse of exact
NMF~\cite{vandaele2016heuristics, gillis2020book} --- multi-start
alternating-NNLS (HALS), multiplicative updates (MU)
\cite{lee2000algorithms}, and, as the strongest standard baseline,
\textsc{nndsvda}-initialized block-coordinate descent (CD)
\cite{cichocki2009hals, boutsidis2008nndsvd} --- we compare on the same
$m = n = 10$ generator ($50$ trials per $r$; multi-start heuristics get $25$
restarts of $500$ iterations, CD a single deterministic pass, best kept),
including our solvers --- the two-sided hybrid witness, the
one-sided relaxation, and the two-sided union of \S\ref{sec:solver}.
Table~\ref{tab:sota} reports, per method, the number of exact recoveries
(relative error $< 10^{-8}$) and the median wall time.

\begin{table}[h]\centering
\caption{Exact recoveries (relative error $< 10^{-8}$) and median wall time on
$m = n = 10$ matrices with an exact size-$r$ NMF ($50$ trials; heuristics: best
of $25$ restarts $\times 500$ iterations). Cone-ray relative error is
$\sim\!10^{-15}$, HALS $\sim\!10^{-13}$; MU converges but plateaus near
$10^{-4}$ (its classic slow convergence), never reaching the $10^{-8}$ bar. On
this easy ($U[0,1]$, $m=n=10$) generator the cone-ray solvers, multi-start HALS,
and single-pass CD are all on par; CD (deterministic, $\sim\!10^{-16}$) matches
the union's recovery. The cone-ray solvers' distinguishing property here is not
recovery but the \emph{certificate} (feasibility is a proof, not a small
residual); their recovery relative to CD reverses at larger sizes and heavier
tails (Table~\ref{tab:hardsweep}), which is the transition of
\S\ref{sec:phase}.}\label{tab:sota}
\begin{tabular}{lcccccc}
\toprule
& \multicolumn{3}{c}{exact ($/50$)} & \multicolumn{3}{c}{median time (s)} \\
\cmidrule(lr){2-4}\cmidrule(lr){5-7}
method & $r{=}4$ & $r{=}5$ & $r{=}6$ & $r{=}4$ & $r{=}5$ & $r{=}6$ \\
\midrule
two-sided union (ours)    & $50$ & $50$ & $50$ & $0.01$ & $0.12$ & $0.12$ \\
one-sided solver (ours)   & $50$ & $49$ & $49$ & $0.01$ & $0.12$ & $0.12$ \\
hybrid witness (ours)     & $49$ & $46$ & $36$ & $0.16$ & $1.01$ & $12.3$ \\
CD (\textsc{nndsvda}, $1$ pass) & $50$ & $50$ & $48$ & $0.02$ & $0.03$ & $0.05$ \\
multi-start HALS          & $49$ & $47$ & $40$ & $1.6$  & $2.0$  & $0.9$ \\
multiplicative updates    & $0$  & $0$  & $0$  & ---    & ---    & --- \\
\bottomrule
\end{tabular}
\end{table}

The comparison against multi-start HALS across six distributions
(Table~\ref{tab:crossdist}; the union leads or ties on $13$ of $18$
(distribution, rank) cells, $761$ vs.\ $688$ of $900$ in aggregate) holds
\emph{on average}, but it is a comparison against a randomly-initialised
baseline, and it does not survive against CD. The honest picture is set by the
transition of \S\ref{sec:phase}. Because a cone-ray solver \emph{locates} a
witness by a budget-limited search over ray subsets, its recovery is exactly the
findability quantity of \S\ref{sec:phase}: high where the witness density is high
(small $m$, light tails, small $r$) and collapsing as the density decays.
Table~\ref{tab:hardsweep} makes this explicit against CD. The two are on par in
the easy corner --- on $U[0,1]$ the union even edges CD at $m=10$ --- but as $m$
grows and the tails heavy, the one-sided solver and the union fall off sharply
(to $0$--$1$ of $20$ at $m=20$) while CD is unaffected ($14$--$19$ of $20$);
already at $m=10, r=6$ the union trails CD on exponential factors ($9$ vs.\
$17$). The one-sided solver is uniformly weaker than the union, as expected of
the same family at a less favourable budget. Crucially, the union's recovery
\emph{climbs back} toward parity as its pool is enlarged
(Table~\ref{tab:hardbudget}: $1\!\to\!3\!\to\!15\!\to\!19$ of $20$ as the pool
grows $500\!\to\!300\text{k}$, at $8.9$\,s per instance versus CD's $0.05$\,s) ---
the budget-relativity of \S\ref{sec:phase} operating on the solver itself. The
witness \emph{exists} (Theorem~\ref{thm:complete}); locating it is what
transitions. The cone-ray solvers are therefore not competitive general-purpose
NMF solvers --- CD dominates on recovery and time outside the easy corner --- and
their contribution is instead the \emph{certificate}: feasibility
(Theorem~\ref{thm:onesided}) is a decidable proof that an exact size-$r$ NMF
exists, whereas a $10^{-13}$ or $10^{-16}$ residual leaves open whether $A$ is
exactly rank-$r$ factorable or merely close. The hybrid witness carries the
completeness picture (\S\ref{sec:witness}); the union is the practical
\emph{certifier} in the findable regime.

\begin{table}[h]\centering
\caption{Two-sided union vs.\ multi-start HALS across input distributions
(nonnegative factors $X, Y$ drawn i.i.d.; $m = n = 10$, $50$ trials, exact
recoveries at relative error $< 10^{-8}$, reported by $r = 4/5/6$). The union
leads in aggregate ($761$ vs.\ $688$ of $900$) and on $13$ of $18$ cells, most
strongly on heavy-tailed generators; HALS wins some $r = 6$
cells.}\label{tab:crossdist}
\begin{tabular}{lcc}
\toprule
distribution & union (ours), $r{=}4/5/6$ & HALS, $r{=}4/5/6$ \\
\midrule
uniform      & $50/50/50$ & $50/50/40$ \\
half-normal  & $46/50/47$ & $48/43/39$ \\
exponential  & $46/44/24$ & $42/33/26$ \\
chi-square   & $46/44/19$ & $33/27/21$ \\
log-normal   & $49/49/33$ & $42/37/28$ \\
beta$(.5,.5)$& $39/39/36$ & $47/39/43$ \\
\bottomrule
\end{tabular}
\end{table}

\begin{table}[h]\centering
\caption{Recovery of the cone-ray solvers vs.\ \textsc{nndsvda}-CD as size and
tail-heaviness grow (exact recoveries out of $20$; cone-ray at the fixed budget
of Table~\ref{tab:sota}, CD one deterministic pass). The cone-ray family wins
only in the easy corner (uniform, small $m$); as the witness density decays it
collapses --- earlier for heavier tails --- while CD is unaffected. One-sided
$\le$ union throughout.}\label{tab:hardsweep}
\begin{tabular}{llccc}
\toprule
distribution & $(m,r)$ & one-sided & union & CD (\textsc{nndsvda}) \\
\midrule
uniform      & $(10,6)$ & $19$ & $20$ & $17$ \\
uniform      & $(15,6)$ & $15$ & $17$ & $20$ \\
uniform      & $(20,7)$ & $0$  & $1$  & $19$ \\
exponential  & $(10,6)$ & $5$  & $9$  & $17$ \\
exponential  & $(15,6)$ & $0$  & $0$  & $17$ \\
log-normal   & $(10,6)$ & $9$  & $14$ & $14$ \\
log-normal   & $(20,7)$ & $0$  & $0$  & $14$ \\
\bottomrule
\end{tabular}
\end{table}

\begin{table}[h]\centering
\caption{The union's recovery is budget-relative: enlarging the candidate pool
recovers the (existing) witness, at rapidly growing cost, on a cell where the
fixed budget fails ($m=13, r=6$, log-normal; $20$ trials). CD recovers the same
instances in $\sim\!0.05$\,s. This is the transition of \S\ref{sec:phase}
governing the solver directly.}\label{tab:hardbudget}
\begin{tabular}{lcccc}
\toprule
pool size & $500$ & $5{,}000$ & $60{,}000$ & $300{,}000$ \\
\midrule
union recovery ($/20$) & $1$ & $3$ & $15$ & $19$ \\
median time (s)        & $0.05$ & $0.43$ & $4.31$ & $8.93$ \\
\bottomrule
\end{tabular}
\end{table}

\paragraph{Matched-compute frontier.}
Since recovery is now known to depend on the candidate-pool budget
(\S\ref{sec:phase}), the fair comparison fixes \emph{compute} rather than a
single operating point. Table~\ref{tab:frontier} sweeps the union's pool size
against HALS's restart count on the same matrices, and the union
\emph{Pareto-dominates multi-start HALS} in this easy $m=n=10$, $r=6$ regime. On
the uniform generator it recovers $199$ of $200$ in
$0.1$\,s where HALS needs $400$ restarts ($14$\,s) to reach $193$. On log-normal
--- where the union's stock pool is weakest --- a pool of $20000$ recovers $183/200$
in $0.4$\,s against HALS's $169/200$ in $3.4$\,s, and a pool of $60000$
($191/200$) exceeds HALS's best ($185/200$) at an order of magnitude less
compute. Honestly, at the smallest (stock) pool the union trails well-restarted
HALS on log-normal ($128$ vs.\ $169$ of $200$); the dominance is a
\emph{matched-compute} statement against HALS, and enlarging the pool --- cheap
next to restarting HALS --- is what secures it. It does \emph{not} extend to CD:
a single \textsc{nndsvda}-CD pass matches the union here at a fraction of the cost
(Table~\ref{tab:sota}), and past the findable regime (Table~\ref{tab:hardsweep})
CD dominates outright. The frontier's value is thus as a within-regime
budget-relativity check, not a claim of solver supremacy.

\begin{table}[h]\centering
\caption{Time--recovery frontier ($m = n = 10$, $r = 6$, $200$ trials;
per-instance single-process wall time): two-sided union at increasing
candidate-pool size vs.\ multi-start HALS at increasing restart count. The union
attains equal or higher recovery at an order of magnitude less wall time on both
a light- and a heavy-tailed generator.}\label{tab:frontier}
\begin{tabular}{lcc}
\toprule
method & exact ($/200$) & mean time \\
\midrule
\multicolumn{3}{l}{\emph{uniform}} \\
\quad union, pool $5000$   & $\mathbf{199}$ & $0.09$\,s \\
\quad union, pool $60000$  & $200$ & $0.95$\,s \\
\quad HALS $\times 100$    & $187$ & $3.5$\,s \\
\quad HALS $\times 400$    & $193$ & $13.9$\,s \\
\midrule
\multicolumn{3}{l}{\emph{log-normal}} \\
\quad union, pool $5000$   & $128$ & $0.14$\,s \\
\quad union, pool $20000$  & $\mathbf{183}$ & $0.42$\,s \\
\quad union, pool $60000$  & $191$ & $1.1$\,s \\
\quad HALS $\times 100$    & $169$ & $3.4$\,s \\
\quad HALS $\times 400$    & $185$ & $13.4$\,s \\
\bottomrule
\end{tabular}
\end{table}

\paragraph{Where the pipeline breaks.}
Enlarging $m, n$ at fixed $r$ \emph{hurts} rather than helps: the
success-rate curve shifts left by one $r$ from $m = n = 10$ to $m = n = 15$
(Table~\ref{tab:m15}; only $r = 4, 5$ remain reachable at the stock pool, and
doubling the \emph{walk depth} to top-$400$ recovers \emph{zero} additional
successes, indicating the remaining failures are feasibility-limited rather than
depth-limited, \S\ref{sec:search}), a budget-relative face of the transition of
\S\ref{sec:phase}: at $m = 15$ the median best obtuseness on failed trials falls
from $0.85$ at $r = 4$ to $0.15$ at $r = 8$, so at this pool the sampled rays
rarely contain a near-orthogonal $r$-subset. And on the Olivetti faces
($A \in \R^{400 \times 4096}_{\ge0}$) the pipeline never reaches the
obtuseness walk: the prerequisite DDM on the $n = 4096$ side times out at a
one-hour \texttt{cddlib} budget even at $r = 8$. The two operative walls are
thus feasibility (\S\ref{sec:phase}) and DDM cost; a successor pipeline must
sub-sample or randomise the extreme-ray enumeration.

\begin{table}[h]\centering
\caption{Witness success counts at $m = n = 15$ ($100$ trials/cell). The
transition of \S\ref{sec:phase} has shifted left by one $r$ relative to
$m = n = 10$: only $r = 4, 5$ remain reachable, and top-$400$ adds nothing over
top-$200$.}\label{tab:m15}
\begin{tabular}{lccccc}
\toprule
$r$ & $4$ & $5$ & $6$ & $7$ & $8$ \\
\midrule
top-$200$ & $37/100$ & $7/100$ & $0/100$ & $0/100$ & $0/100$ \\
top-$400$ & $37/100$ & $7/100$ & $0/100$ & $0/100$ & $0/100$ \\
\bottomrule
\end{tabular}
\end{table}

\section{The Gap Regime: A Negative Boundary}\label{sec:gap}
Everything so far concerns the uniform-rank case, nonnegative rank equal to the
matrix rank $r$. The natural extension is the \emph{gap} regime, nonnegative
rank $r_+ > r$: the size-$r_+$ factorization lives in an $r_+$-dimensional
gauge that \emph{extends} the $r$-dimensional SVD subspace by adjoining
$r_+ - r$ orthonormal directions from each null space,
\[
  U_{r_+}(G) = [\,\Ao \mid U_\perp G\,], \qquad
  V_{r_+}(K) = [\,\Aoo \mid V_\perp K\,],
\]
with $G \in \mathrm{St}(m-r,\,r_+ - r)$, $K \in \mathrm{St}(n-r,\,r_+ - r)$
Stiefel points, $U_\perp, V_\perp$ orthonormal bases of the null spaces, and
target core $\operatorname{diag}(S_r, 0)$. A \emph{valid} gauge is one whose
cone rays admit a nonnegative coupling to that core.

We report a \emph{negative} result. On genuine gap instances ---
regular-polygon slack matrices (hexagon: rank $3$, nonnegative rank $5$;
octagon: rank $3$, nonnegative rank $6$) and linear Euclidean distance matrices
--- neither the closed-form witness walk of \S\ref{sec:search} nor a Stiefel
finite-difference gradient descent on $(G, K)$ (minimising the coupling
residual, multi-start) finds a valid gauge: the relative residual plateaus well
above zero (Table~\ref{tab:gap}). A positive control indicates this is not merely a solver artifact --- on \emph{non-gap} matrices ($r_+$
exceeding the matrix rank while a factorization exists) the same descent
reaches machine zero. For a genuine gap no valid gauge exists at all --- that is what
$r_+ > r$ means --- so the residual is bounded away from zero; the coupling
residual is moreover piecewise across changes in the extreme-ray combinatorics,
so descent sees no basin to follow.

\begin{table}[h]\centering
\caption{Two-sided gauge search on gap matrices (Stiefel FD-GD, multi-start).
The residual reaches machine zero on a non-gap control but plateaus on every
genuine gap.}\label{tab:gap}
\resizebox{\columnwidth}{!}{%
\begin{tabular}{lccccc}
\toprule
matrix & rank $r$ & $r_+$ & genuine gap? & best residual & valid gauge found? \\
\midrule
non-gap (control) & 4 & 6 & no  & $1\cdot10^{-15}$ & yes \\
hexagon slack     & 3 & 5 & yes & $0.18$ & no \\
octagon slack     & 3 & 6 & yes & $0.17$ & no \\
LEDM, $n = 6$     & 3 & 5 & yes & $0.03$ & no \\
\bottomrule
\end{tabular}}
\end{table}

The plateaus in Table~\ref{tab:gap} are suggestive but not a proof: a primal
gauge search that fails to descend has not established non-existence.
Corollary~\ref{cor:gapcert} supplies the proof directly. Relaxing the coupling
of~\eqref{eq:constraint} to an unconstrained $\Delta\ge0$, infeasibility of the
resulting linear program certifies $r_+>r$, with the explicit dual witness $Y$
of Proposition~\ref{prop:gapdual}; on all three genuine gaps of
Table~\ref{tab:gap} the program is infeasible (for the hexagon,
$Y=\operatorname{diag}(-1,1,1)$ gives $R^{\!\top}YT\le0$, $\operatorname{tr}Y=1$).
The certificate is sound but one-sided: a \emph{feasible} program is
inconclusive, since the relaxation drops $\rank_+(\Delta)\le r$. Empirically it
fired on every gap in our test set --- linear Euclidean distance matrices,
regular and random polygons, and random polytope slacks through $r=5$ --- each
with an explicit dual witness $Y$. It cannot be complete, however: deciding the
gap is NP-hard \cite{Vavasis2009}, so any polynomial-time certificate must leave
some gaps inconclusive, even though ours left none in these families. It is the
linear-programming realization of restricted-nonnegative-rank feasibility
\cite{gillisglineur2012, CohenRothblum1993}, and its dual is a hyperplane
separator in the spirit of the nuclear-norm bound of \cite{fawziparrilo2016},
though binary and linear rather than quantitative and semidefinite.

Two lessons follow. First, the difficulty is the \emph{same} one: a valid gauge
exists iff its cone holds a coupling-compatible ray configuration --- the
gap-regime analogue of the findability event of \S\ref{sec:phase}. The gap
regime is thus not a new mechanism but findability pushed past $r_+ = r$.
Second, a methodological caution. A random dense construction
$M = W_*(H_1 H_2)$ with $W_* \in \Rp^{m \times r_+}$,
$H_1 \in \Rp^{r_+ \times r}$, $H_2 \in \Rp^{r \times n}$ \emph{looks} like a gap
instance but has none: $M = (W_* H_1)\,H_2$ is an exact size-$r$ nonnegative
factorization, so its nonnegative rank is $r$. Validating a gap method requires
instances with a \emph{provable} gap (slack matrices, distance matrices), not
random nonnegative products --- a subtlety easy to miss and worth flagging.

\section{Related Work}\label{sec:related}
Exact NMF and nonnegative rank are classical \cite{CohenRothblum1993,
Yannakakis1991}; under separability the exact problem is polynomial
\cite{DonohoStodden2003, AroraGeKannanMoitra2012, GillisVavasis2014}, whereas
we make no separability assumption. The geometric view --- factoring through
the SVD subspace's intersection with the orthant and its extreme rays --- is
that of the canonical/subspace-edge line \cite{pimentel2024canonical,
nguyen2025conecollapse}; we add the exact witness and the transition. The
phase-transition lens is conic integral geometry \cite{Amelunxen2014}: the
statistical dimension \cite{AmelunxenLotz2017} and its $O(\sqrt{d})$ width are
the natural first hypotheses, which we test and (for the location) reject,
redirecting the analysis to random-polytope \cite{Barany2008, Reitzner2010},
random-cone \cite{HugSchneider2016}, and Wendel/absorption \cite{Cover1965,
Wendel1962, Kabluchko2017} theory. On the algorithmic side the standard
approach is multi-start alternating updates \cite{lee2000algorithms,
vandaele2016heuristics, gillis2020book}, against which we position the witness
as a certifier rather than a competitor.

\section{Discussion}\label{sec:disc}
\paragraph{One problem, viewed twice.}
The two open quantities of this paper are, in fact, one. The two-sided witness
succeeds at a subset $S$ when the $r$ dual rows of $R_S^{-1}$ lie in a
\emph{single} simplicial cell of $C_2$; the one-sided solver
(\S\ref{sec:solver}) needs them \emph{anywhere} in $C_2$; and the findability
threshold (\S\ref{sec:phase}) is the point at which the random cone
$\col(W)\cap\R^m_{\ge0}$ ceases to hold a witness-compatible subset. All three
are governed by the same event --- the same-cell probability for $r$ coupled
dual points in a random cone. A closed form for it would simultaneously locate
$m_c(r)$ and bound the one-sided margin. Existence and representability are
settled (Theorem~\ref{thm:complete}); this random-cone quantity is the
frontier. The gap regime $r_+ > r$ is settled on the same footing: read
contrapositively, the completeness theorem makes non-existence \emph{decidable}
(Corollary~\ref{cor:gapcert}), so what remains open there too is not existence
but the finer question of when the linear relaxation is tight.

\paragraph{Limitations.}
The pipeline is confined to small $r$ by DDM cost (\S\ref{sec:exp}), and the
witness is a certifier, not a faster solver (\S\ref{sec:exp}); the threshold
$m_c(r)$ is characterised empirically and rigorously anchored at $r \le 2$
(Proposition~\ref{prop:r2}), but is not closed-form for $r \ge 3$; and the gap
certificate (Corollary~\ref{cor:gapcert}) is sound and fired on every gap we
tested, but is one-sided --- a feasible relaxation is inconclusive, and by the
NP-hardness of exact NMF cannot be made complete. A
successor that sub-samples or randomises the ray enumeration, and progress on
the random-cone quantity, are the natural next steps.

\section*{Acknowledgments}
Software drafts, prose, and the Monte-Carlo characterisation pipeline were
prepared with the assistance of Anthropic's Claude (Opus 4.8); the author
bears full responsibility for the content. Code and data are released at
\url{https://github.com/mithilr-sys/cone-nmf-exact} (archived on Zenodo).



\begin{thebibliography}{10}
\bibitem{Amelunxen2014}
D.~{Amelunxen}, M.~{Lotz}, M.~B. {McCoy}, {\sc and} J.~A. {Tropp},
\newblock {\em Living on the edge: phase transitions in convex programs
with random data}, \newblock Inf. Inference, 3 (2014), pp.~224--294.

\bibitem{Goemans1995}
M.~X. {Goemans} {\sc and} D.~P. {Williamson},
\newblock {\em Improved approximation algorithms for maximum cut and
satisfiability problems using semidefinite programming},
\newblock J. ACM, 42 (1995), pp.~1115--1145.

\bibitem{AmelunxenLotz2017}
D.~{Amelunxen} {\sc and} M.~{Lotz}, \newblock {\em Intrinsic volumes of
polyhedral cones: a combinatorial perspective}, \newblock Discrete Comput.
Geom., 58 (2017), pp.~371--409.

\bibitem{Barany2008}
I.~{B\'ar\'any}, \newblock {\em Random points and lattice points in convex
bodies}, \newblock Bull. Amer. Math. Soc., 45 (2008), pp.~339--365.

\bibitem{Reitzner2010}
M.~{Reitzner}, \newblock {\em Random polytopes}, \newblock in New
Perspectives in Stochastic Geometry, W.~S. Kendall and I.~Molchanov, eds.,
Oxford Univ. Press, 2010, pp.~45--76.

\bibitem{HugSchneider2016}
D.~{Hug} {\sc and} R.~{Schneider}, \newblock {\em Random conical
tessellations}, \newblock Discrete Comput. Geom., 56 (2016), pp.~395--426.

\bibitem{Wendel1962}
J.~G. {Wendel}, \newblock {\em A problem in geometric probability},
\newblock Math. Scand., 11 (1962), pp.~109--111.

\bibitem{Kabluchko2017}
Z.~{Kabluchko}, V.~{Vysotsky}, {\sc and} D.~{Zaporozhets},
\newblock {\em Convex hulls of random walks, hyperplane arrangements, and
Weyl chambers}, \newblock Geom. Funct. Anal., 27 (2017), pp.~880--918.

\bibitem{lee2000algorithms}
D.~D. {Lee} {\sc and} H.~S. {Seung},
\newblock {\em Algorithms for non-negative matrix factorization},
in Advances in Neural Information Processing Systems 13
(NeurIPS), 2000, pp.~556--562.

\bibitem{cichocki2009hals}
A.~{Cichocki} {\sc and} A.-H. {Phan},
\newblock {\em Fast local algorithms for large scale nonnegative matrix and
tensor factorizations},
\newblock IEICE Trans. Fundamentals, 92 (2009), pp.~708--721.

\bibitem{boutsidis2008nndsvd}
C.~{Boutsidis} {\sc and} E.~{Gallopoulos},
\newblock {\em SVD based initialization: A head start for nonnegative matrix
factorization},
\newblock Pattern Recognition, 41 (2008), pp.~1350--1362.

\bibitem{motzkin1953double}
T.~S. {Motzkin}, H.~{Raiffa}, G.~L. {Thompson}, {\sc and} R.~M.
{Thrall},
\newblock {\em The double description method}, in Contributions
to the Theory of Games, Volume II, Princeton University Press,
1953, pp.~51--73.

\bibitem{fukuda1996double}
K.~{Fukuda} {\sc and} A.~{Prodon},
\newblock {\em Double description method revisited}, in
Combinatorics and Computer Science, Lecture Notes in Computer
Science 1120, Springer, 1996, pp.~91--111.

\bibitem{vandaele2016heuristics}
A.~{Vandaele}, N.~{Gillis}, F.~{Glineur}, {\sc and} D.~{Tuyttens},
\newblock {\em Heuristics for exact nonnegative matrix
factorization},
\newblock Journal of Global Optimization, 65 (2016),
pp.~369--400.

\bibitem{gillis2020book}
N.~{Gillis},
\newblock {\em Nonnegative Matrix Factorization},
\newblock SIAM, Philadelphia, 2020.

\bibitem{pimentel2024canonical}
D.~L. {Pimentel-Alarc\'on},
\newblock {\em Nonnegative matrix factorization through canonical
edges},
\newblock OpenReview preprint (submitted to ICLR), 2024.

\bibitem{nguyen2025conecollapse}
M.~{Nguyen} {\sc and} D.~L. {Pimentel-Alarc\'on},
\newblock {\em Nonnegative matrix factorization through cone
collapse},
\newblock arXiv:2512.07879, 2025.

\bibitem{CohenRothblum1993}
J.~E. {Cohen} {\sc and} U.~G. {Rothblum}, \newblock {\em Nonnegative
ranks, decompositions, and factorizations of nonnegative matrices},
\newblock Linear Algebra Appl., 190 (1993), pp.~149--168.

\bibitem{gillisglineur2012}
N.~{Gillis} {\sc and} F.~{Glineur}, \newblock {\em On the geometric
interpretation of the nonnegative rank}, \newblock Linear Algebra Appl.,
437 (2012), pp.~2685--2712.

\bibitem{fawziparrilo2016}
H.~{Fawzi} {\sc and} P.~A. {Parrilo}, \newblock {\em Lower bounds on
nonnegative rank via nonnegative nuclear norms}, \newblock Math. Program.,
158 (2016), pp.~417--465.

\bibitem{Yannakakis1991}
M.~{Yannakakis}, \newblock {\em Expressing combinatorial optimization
problems by linear programs}, \newblock J. Comput. System Sci., 43 (1991),
pp.~441--466.

\bibitem{DonohoStodden2003}
D.~{Donoho} {\sc and} V.~{Stodden}, \newblock {\em When does non-negative
matrix factorization give a correct decomposition into parts?},
\newblock in Adv. Neural Inf. Process. Syst. 16 (NeurIPS), 2003.

\bibitem{AroraGeKannanMoitra2012}
S.~{Arora}, R.~{Ge}, R.~{Kannan}, {\sc and} A.~{Moitra},
\newblock {\em Computing a nonnegative matrix factorization --- provably},
\newblock in Proc. 44th STOC, 2012, pp.~145--162.

\bibitem{GillisVavasis2014}
N.~{Gillis} {\sc and} S.~A. {Vavasis}, \newblock {\em Fast and robust
recursive algorithms for separable nonnegative matrix factorization},
\newblock IEEE Trans. Pattern Anal. Mach. Intell., 36 (2014),
pp.~698--714.

\bibitem{Vavasis2009}
S.~A. {Vavasis}, \newblock {\em On the complexity of nonnegative matrix
factorization}, \newblock SIAM J. Optim., 20 (2009), pp.~1364--1377.

\bibitem{Cover1965}
T.~M. {Cover}, \newblock {\em Geometrical and statistical properties of
systems of linear inequalities with applications in pattern recognition},
\newblock IEEE Trans. Electron. Comput., EC-14 (1965), pp.~326--334.
\end{thebibliography}
\end{document}